\documentclass[10pt,a4paper,reqno]{amsart}  
  \usepackage{cite}  
\usepackage[english]{babel}
\usepackage{amssymb,amsthm}  
\usepackage{verbatim,latexsym,mathrsfs,color}    
\usepackage{varioref,amsfonts}

\def\mybig#1{\text{\Large$#1$}}

\usepackage[bookmarks=true,%
    colorlinks=true,%
    linkcolor=blue,%
    citecolor=blue,%
    filecolor=blue,%
    menucolor=blue,%
    urlcolor=blue,%
    breaklinks=true]{hyperref}

\usepackage{tikz-cd}

 \def\L{\mathbf L}   
\def\be{\begin{equation}}  \def\ee{\end{equation}}
\def\bal{\begin{aligned}}   \def\eal{\end{aligned}}
\def\bes{\begin{equation*}}  \def\ees{\end{equation*}}
\def\narrower{%
  \advance\leftskip\parindent
  \advance\rightskip\parindent}

\def\be{\begin{equation}}   \def\ee{\end{equation}}     
\def\bes{\begin{equation*}}    \def\ees{\end{equation*}}
\def\ba{\be\begin{aligned}} \def\ea{\end{aligned}\ee}   
\def\bas{\bes\begin{aligned}}  \def\eas{\end{aligned}\ees}

\def\={\;=\;}  \def\:{\;:=\;}  \def\+{\,+\,}  \def\m{\,-\,}  \def\i{^{-1}}
  \def\ssm{\smallsetminus}    \def\lr#1{\langle#1\rangle}   
\def\wt{\widetilde} \def\wh{\widehat} \def\p{\partial}	
\def\thin{{\hskip 1pt}} 

 \def\Z{\mathbb Z} \def\R{\mathbb R}  \def\C{\mathbb C}    
       \def\H{\mathbb H}
\def\G{\Gamma} \def\t{\tau}   \def\g{\gamma} \def\de{\delta} \def\om{\omega}
 \def\Om{\Omega} \def\D{\Delta} \def\th{\theta}  \def\vth{\vartheta}
   \def\g{\gamma} \def\G{\Gamma}  \def\s{\sigma}
\def\k{k} \def\l{\lambda}  \def\t{\tau}   
\def\tM{\wt M} \def\hM{\wh M}     
      
     \def\Dh{\wh D}   \def\Eh{\wh E_2}  
\def\K{\mathcal K} \def\K{\mathfrak K} \def\K{\mathbb K}  \def\K{\Theta} 
\def\pih{\pi^{\text{\rm hol}}}  \def\pihh{\wh\pi}  \def\pip{\pi}  
\def\sl#1{\bigl|_{#1}}

\def\DD#1{D\up{#1}}

\def\MLDO{{\rm MLDO}}  \def\MLDO{{\textsf{MLDO}}}   \def\MLDO{{\sf{MLDO}}}

     \def\ra {\rightarrow} 
\newcommand{\Ker}{\operatorname{Ker}}
\def\SD{\mathfrak d}
\def\d{\SD}

\def\sd{\mathfrak{d}}
\def\MG{SL_2(\Z)}
\def\SL{SL_2(\R)}

\theoremstyle{plain}
   \newtheorem{theorem}{Theorem} 

   \newtheorem{lemma}{Lemma}
   \newtheorem{proposition}{Proposition}
   
\newtheorem*{theorem*}{Theorem}
 \newtheorem*{corollary*}{Corollary}
 \newtheorem*{lemma*}{Lemma}
 
\theoremstyle{remark}
   
   \newtheorem{remark}{Remark}

\def\lie{\mathfrak{sl}_2}
\def\Hh{\mathfrak H}
\def\H{\mathfrak H}
\def\Hol{\mathsf{Hol}}  \def\AHol{\mathsf{AHol}}
 
\def\qM{\wt M}      
\def\E{\bold E} \def\E{\mathcal{E}_2} \def\E{\Phi}  \def\E{\mathcal{E}}

\def\bt{\bar\t}

\def\dM{M^D_*(\G)}  \def\dMp{M^D_{>0}(\G)}
\def\up#1{^{\langle#1\rangle}}

\newcommand{\sm}[4]{\bigl(\smallmatrix #1&#2\\ #3&#4\endsmallmatrix\bigr)}
\newcommand{\bm}[4]{\Bigl(\begin{matrix} {\,#1}&{#2\,}\\ {\,#3}&{#4\,}\end{matrix}\Bigr)}

\newcommand{\bprf}{\noindent\textit{Proof.}\;}
\newcommand{\eprf}{$\qquad\square$\medskip}
\newcommand{\bl}{\begin{lemma}}
\newcommand{\el}{\end{lemma}}
\newcommand{\bp}{\begin{proposition}}
\newcommand{\ep}{\end{proposition}}
\newcommand{\bpr}{\begin{proof}}
\newcommand{\epr}{\end{proof}}
\newcommand{\br}{\begin{remark}}
\newcommand{\ebr}{\end{remark}}

\begin{document}

\bibliographystyle{abcd}

\title[Modular Linear Differential Operators]
  {Modular Linear Differential Operators \\ and Generalized Rankin-Cohen Brackets}
\author{Kiyokazu Nagatomo}  
\address{Osaka University, Osaka 565-0871, Japan}
\email{nagatomo@math.sci.osaka-u.ac.jp}
\author{Yuichi Sakai}  
\address{Kurume Institute of Technology, Fukuoka 830-0052, Japan}
\email{dynamixaxs@gmail.com}
\author{Don Zagier}  
\address{Max Planck Institute for Mathematics,  53111 Bonn, Germany\newline \hphantom{Xx} 
 International Centre for Theoretical Physics, Trieste, Italy }
\email{dbz@mpim-bonn.mpg.de}
\maketitle

\setlength{\parskip}{3pt}
\setlength{\parindent}{15pt}

\bigskip
\noindent

\bigskip
\noindent

{\narrower {\small \bf Abstract. \rm  
The aim in this paper is to give expressions for modular linear differential operators of any order. 
In particular, we show that they can all be described in 
terms of Rankin-Cohen brackets and a modified Rankin-Cohen bracket found by Kaneko and Koike.
We also  give more uniform descriptions of MLDOs in terms of canonically defined higher Serre derivatives 
and an extension of Rankin-Cohen brackets, as well as in terms of quasimodular forms and 
almost holomorphic modular forms.  The last of these descriptions involves the holomorphic
projection map.  The paper also includes some general results on the theory of quasimodular forms 
on both cocompact and non-cocompact subgroups of~$\SL$, as well as a slight sharpening of 
a theorem of Martin and Royer on Rankin-Cohen brackets of quasimodular forms}. \par} 
\vskip 0.7 cm

 \setcounter{tocdepth}{1}          

{\footnotesize
\tableofcontents
}

\section{Introduction}  \label{intro}

Modular linear differential operators (MLDOs) and the corresponding
modular linear differential equations (MLDEs) have appeared in recent years
in a variety of contexts, ranging from the study of supersingular $j$-invariants to
the classification of vertex operator algebras in terms of the spaces of
modular forms spanned by the characters of their irreducible modules. In this 
paper we study and describe these operators and these differential equations in
several different ways:
\begin{itemize}
\item as ordinary differential operators $L=\sum_{r=0}^n a_r(\t)\,D^r$, where $D$ 
is the normalized differentiation operator (see below) and the $a_r(\t)$ are quasimodular 
forms that have to satisfy certain auxiliary conditions in order to make the operator $L$ modular;
\item as operators of the form $L=\sum_{r=0}^n b_r(\t)\,\d^r$, where $\d^r$ is the $r$\thin th 
iterate of the Serre derivative~$\d$ (see below) and the $b_r(\t)$ are modular forms;
\item as linear combinations of two other types of higher-order Serre derivatives;
\item in terms of Rankin-Cohen brackets and the Kaneko-Koike operator (defined below); 
\item uniformly in terms of extended Rankin-Cohen brackets; 
\item in terms of quasimodular forms and the action of $\lie$ on the space of quasimodular forms;
\item in terms of almost holomorphic modular forms and the holomorphic projection operator.
\end{itemize}
Each of these approaches leads to a complete description of all~MLDEs, but the representations 
obtained are quite different and are related to one another in non-obvious ways.
The most uniform of these are two canonical bijections between the spaces of MLDOs and of quasimodular 
forms on any lattice in~$\SL$, respecting both the weights and the natural filtrations on these 
two spaces, that are given in Section~\ref{MLDOandQMF} and in Section~\ref{MLDOandPrimProj} 
(or, in terms of almost holomorphic modular forms, in Sction~\ref{MLDOandHolProj}).

The paper falls naturally into two parts. In Section~\ref{basics}--\ref{structure} we will recall some basic
definitions, give the precise definition of MLDEs and simple properties of their solution spaces, describe a number
of concrete examples of MLDOs including the Rankin-Cohen brackets and three kinds of higher-order Serre derivatives,
and state the main structure theorems describing all MLDOs in the case of the full modular group, as well as
giving examples coming from characters of modules over vertex operator algebras. The remaining sections give 
refinements, proofs, and extensions to other lattices (with a fairly detailed discussion of the 
different forms that the theory takes for cocompact and non-cocompact lattices) and also more conceptual 
descriptions of the various isomorphisms in terms of ``extended Rankin-Cohen brackets" and of almost holomorphic
modular forms.  We also give  as an application a result saying that one of the three kinds of higher-order 
Serre derivatives can be modified slightly to act on the space of quasimodular forms with a given upper bound on their
depth (Theorem~\ref{VZforQMF}), and as a corollary of this a slight sharpening of a result of Martin and 
Royer~\cite{MR} saying that a similar modification of Rankin-Cohen brackets also acts on pairs of quasimodular 
forms without increasing their combined depths.

\section{Review of basic definitions}  \label{basics}

In this preliminary section we review some basic notions, including modular forms, quasimodular forms 
for the full modular group, and the Serre derivative.  This material will be familiar to most readers but is 
included anyway for completeness and to fix notations.

By {\it lattice} we will mean a discrete subgroup~$\G$ of finite covolume in~$\SL$, acting in the complex upper
half-plane~$\H$ by fractional transformations, i.e.,~$\g=\sm abcd$ sends $\t\in\H$ to $\g\t=\frac{a\t+b}{c\t+d}$. 
The case of most interest is when $\G$ is the full modular group~$\MG$, which we will denote 
by~$\G_1$, but in the latter sections of the paper will consider other lattices as well, including cocompact
ones whose fundamental domains are compact in~$\H$ and for which there are no cusps, no Eisenstein series,
and no Serre derivatives.  If $\G$ is an arbitrary lattice, then for $k\in\Z$ we denote by $M_k(\G)$ the 
space of holomorphic modular forms of weight~$k$ on~$\G$, meaning holomorphic functions~$f$ on~$\H$ 
satisfying $f|_k\g=f$ for all $\g\in\G$. Here $f\mapsto f|_k\g$ is the operation 
(``slash operator") defined for $\g=\sm abcd\in\SL$ by
\be\label{Slash}  \bigl(f|_k\g\bigr)(\t)\=(c\t+d)^{-k}f\Bigl(\frac{a\t+b}{c\t+b}\Bigr)
\ee
and the word ``holomorphic" includes a growth condition at cusps that is standard and will not be repeated here.  
(For~$\G=\G_1$ it is just the condition that $f(\t)$ is bounded as $\Im(\t)\to\infty$.)  The space $M_k(\G)$
is finite-dimensional over~$\C$ for each~$k$ and the direct sum $M_*(\G):=\bigoplus_kM_k(\G)$ is a finitely 
generated graded $\C$-algebra.  When $\G=\G_1$ we usually omit it from the notations, writing simply $M_*$
for $M_*(\G_1)$.  It is given explicitly by $M_*=\C[E_4,E_6]$, where $E_k$ is the standard normalized Eisenstein series 
\be\label{EisDef} E_k(\t) \= 1 \+ C_k\,\sum_{n=1}^\infty n^{k-1}\,\frac{q^n}{1-q^n}\qquad(k=2,\,4,\,6,\,\dots)\,. \ee
(Here and from now on $q=e^{2\pi i\t}$, and the $C_k$ are well-known rational numbers with $C_2=-24$, $C_4=240$, 
and $C_6=-504$.)  The form $E_2$ is not modular, but belongs to the larger 
$\C$-algebra $\tM_*=\C[E_2,E_4,E_6]$ of quasimodular forms on~$\G_1$, and for the moment we simply take this
as the definition of~$\tM_*$.  (The definition and properties of quasimodular forms for arbitrary 
lattices~$\G$ will be given in~Section~\ref{MLDOandQMF}.)  The important point is that the derivative of a 
modular form of positive weight is never modular, but is always quasimodular. In fact the ring 
$\tM_*(\G)$ (for any~$\G$) is closed under differentiation for any~$\G$, as one sees explicitly 
in the case of the full modular group from Ramanujan's famous formulas
\be \label{RamDer}
 E_2'\=\frac{E_2^2-E_4}{12}\,, \qquad E_4'\=\frac{E_2E_4-E_6}3\,,\qquad E_6'\=\frac{E_2E_6-E_4^2}2\,. \ee
(Here and from now on we denote by $f'(\t)$ or by $Df(\t)$ the renormalized derivative 
$(2\pi i)\i\,d/d\t$, where the factor $(2\pi i)\i$ is included so that the operator $D=q\,d/dq$ 
preserves the space of power series with rational coefficients in $q$.) These formulas show that the
derivative $Df=f'$ of a quasimodular form of weight~$k$ has weight~$k+2$ in this case, and the same holds for
quasimodular forms on any lattice.  Another consequence of~\eqref{RamDer}, which will play an important
role in this paper, is that the {\it Serre derivative} 
\be \label{Serrederiv}
  \d_k f(\t) \: f'(\t) \m \frac k{12}\,E_2(\t)\,f(\t) \,, \ee
maps $M_k$ to $M_{k+2}$.   Since a modular form has a well-defined weight, we will often omit the index~$k$. 
In particular, we will often write $\d^n$ (or by abuse of notation $\d_k^n$) for the iterated Serre derivative 
$\d_{k+2n-2}\cdots\d_{k+2}\d_k$, which maps modular forms of weight~$k$ to modular forms of weight~$k+2n$.

\section{Modular linear differential equations and their solution spaces}  \label{SolSpace}

We can now formulate the key definition.  A {\it modular linear differential operator} (MLDO) of {\it weight}~$K$
and {\it type}~$(k,k+K)$ on~$\G$ is a linear differential operator $L$ of finite order, with holomorphic coefficients
(also at the cusps, in the usual sense), satisfying 
\be\label{MLDEdef}  L\bigl(f|_k\g\bigr) \= L(f)|_{k+K}\g \ee
for all holomorphic functions $f$ in the upper half-plane and all~$\g\in\G$. The corresponding
{\it modular linear differential equation} (MLDE) is then the linear differential equation~$Lf=0$. 
Notice that in this definition, $k$ can be positive or negative and in fact need not even be an integer, 
but~$K$, as we will see, will always be a non-negative integer if the operator~$L$ is non-zero, and 
in fact always strictly positive except in the uninteresting case when $L$ is multiplication by a constant.

Equation~\eqref{MLDEdef} implies in particular that $L$ maps $M_k(\G)$ to~$M_{k+K}(\G)$. But
often we will apply the operator $L$ to some space of holomorphic or meromorphic modular forms of 
weight~$k$ on some smaller subgroup $\G'\subset\G$. The important remark is that the kernel of $L$ is 
a finite-dimensional  (of dimension equal to the order~$n$ of~$L$) and invariant under the action
of~$\G$ in weight~$\k$.  Conversely, any vector space~$V$ having these two properties is the solution
space of some MLDE.  To see this, we choose a basis $f_1,\,f_2,\,\dots,\,f_n$ of~$V$ and define a 
differential operatoMax Planck Institute for Mathematics, 53111 Bonn, Germany
Email address: dbz@mpim-bonn.mpg.de
r~$\D_V$ of order~$n$ (depending up to a scalar factor depends only on~$V$ and 
not on the chosen basis) by
$$ \D_V(f) \;\,=\;\, \begin{vmatrix} f&\d_\k f&\cdots&\d_\k^nf\\ f_1&\d_\k f_1&\cdots&\d_\k^nf_1\\
            \vdots&\vdots&&\vdots\\ f_n&\d_\k f_n&\cdots&\d_\k^nf_n \end{vmatrix}
\;\,=\;\, \begin{vmatrix} f&D f&\cdots&D^nf\\ f_1&D f_1&\cdots&D^nf_1\\
      \vdots&\vdots&&\vdots\\ f_n&D f_n&\cdots&D^nf_n  \end{vmatrix}\,, $$
where $\d_k^r=\d_{k+2r-2}\cdots\d_{k+2}\d_k$ is the $r$\thin th iterate of the Serre derivative~\eqref{Serrederiv} 
and where the equality of the two determinants follows from the inductively proved fact 
that each operator $\d_k^r$ is a the sum of $D^r$ and a linear combination of $D^p$ 
with $p<r$ and with quasimodular coefficients (of weight~$2r-2p$) depending only 
on~$\k$, $r$~and~$p$.  Expanding the determinants by their first rows, we obtain two expressions 
  $ \D_V=\sum_{r=0}^n a_r(\t)\,D^r$ and $\D_V=\sum_{r=0}^n b_r(\t)\,\d^r$
for the operator~$\D_V$ in which each $a_r$ is a quasimodular form of weight $n(\k+n+1)-2r$
and depth~$\le n-r$, while the $b_r$ are modular forms of the same weights.
The second expression together with the transformation property $\d_k^r(f|_k\g)=(\d_k^rf)|_{\k+2r}\g$
shows that $\D_V$ is indeed an MLDO of type $(\k,n(\k+n+1))$.

We make several remarks about this construction.  First of all, if $V$~is the 
solution space of $Lf=0$ for some MLDO~$L$, then $L$ and $\D_V$ needn't agree up to a scalar factor,
but may differ by a non-constant left factor, and indeed this often happens.  An example is given by 
the space~$V$ spanned by the Rogers-Ramanujan functions defined by~\eqref{RReqs}. As we will discuss in the 
next section, $V$ is the solution space of the operator~$L_{2,1/5}$ defined in~\eqref{KZDO}, but when we calculate 
$\D_V$ with respect to the basis $(G_0,G_1)$ we find that $\D_V$ is equal to $\frac15\eta(\t)^4$ times $L_{2,1/5}$,
where $\eta(\t)=\D(\t)^{1/24}=q^{1/24}\prod(1-q^n)$ as usual.  One way to normalize the MLDO having
a given space~$V$ as its solution space is to fix the top coefficient (e.g.~taking it to be~1,
which is the case of {\it monic} MLDOs discussed below, although making this normalization for a general
MLDO with holomorphic coefficients may lead to an MLDO with only meromorphic coefficients). Another is to
normalize~$L$ so that the quasimodular or modular forms occurring as coefficients in its $D$- or $\d$-expansion 
are holomorphic and have no common factor of positive weight.  In ideal cases, like for the 
operator~\eqref{KZDO} below, these agree.  But in general one should be aware that giving a modular 
linear differential equation is the same as giving its solution space, but is not quite the same as giving 
a modular linear differential operator, since the opeators $L$ and $hL$ for any function $h(\t)$ give the same 
differential equation $Lf=0$, and for some purposes it is important to keep track of this distinction.

\section{Examples of modular linear differential operators of small orders} \label{examples}


The simplest non-trivial example of an MLDO is the Serre derivative~$\d_k$ as defined in~\eqref{Serrederiv}.
It has order~1, weight~2, and type~$(k,k+2)$ and up to a constant factor is the only MLDO with these parameters.
Notice that $k$ here can be positive or negative, and that for~$k=0$ the Serre derivative reduces to just~$D$,
which is indeed an MLDO of type~$(0,2)$.

The next example, which has order~2 and weight~4, arose in~\cite{KZ} in connection with the study of supersingular 
$j$-invariants in characteristic $p>3$.  (The definition of supersingular plays no role here and will not be recalled.)
If $k$ is a positive even integer, then the composition $\d_k^2$ of $\d_k$ with $\d_{k+2}$ goes from $M_k$ to $M_{k+4}$.
If $k+4$ is not divisible by~3 (which is always true in the application to supersingular $j$-values, where $k=p-1$
with $p>3$ prime), then $M_{k+4}=M_k\cdot E_4$ and we obtain an endomorphism $E_4\i\d_k^2=E_4\i\d_{k+2}\d_k$ of the 
finite-dimensional space $M_k$. Its eigenvalues turn out to be $\l_{k-12n}$ for $0\le n\le k/12$, where 
$\l_k=\frac{k(k+2)}{144}$.  The eigenfunction $F_k(\t)$ with eigenvalue~$\l_k$ is unique up to a scalar and 
can be given explicitly in terms of the Euler-Gauss hypergeometric function ${}_2F_1(a,b;c;x)$ by
 $$ F_k(\t) \= E_4(\t)^{k/4}{}_2F_1\Bigl(-\frac k{12},-\frac{k-4}{12};-\frac{k-5}6;\frac{1728}{j(\t)}\Bigr)\,, $$
 where $j(\t)=E_4(\tau)^3/\D(\t)$ is the usual  modular invariant 
 (here $\D(\t)=q\prod_{n\ge1}(1-q^n)^{24}=(E_4(\t)^3-E_6(\t)^2)/1728\in M_{12}$, the Ramanujan discriminant function), 
and the eigenfunction with eigenvalue $\l_{k-12n}$ is then~$\D(\t)^nF_{k-12n}(\t)$. 
In particular the function $F_k(\t)$ is a solution (and in fact the unique solution in~$M_k$ up to a scalar) 
of the MLDE $L_{2,k}(f)=0$, where $L_{2,k}$ is the MLDO
\be\label{KZDO}  L_{2,k} \= \d_k^2\m\l_k\,E_4 \ee  
of order~2 and type $(k,k+4)$. Written out explicitly in terms of $D$ rather than~$\d$, this differential
equation takes the form
\be\label{KZ}  f''(\t) \m \frac {k+1}6\,E_2(\t)f'(\t) \+ \frac{k(k+1)}{12}\,E_2'(\t)\,f(\t) \= 0 \;. \ee  

Our third example has order~3.  Recall that the ``Thetanullwerte" $\th_j(\t)$ ($j=2,3,4$) are the values
at $z=0$ of the Jacobi theta functions $\th_j(\t,z)$ (the function $\th_1(\t,z)$ vanishes at~$z=0$), given
explicitly as $\sum q^{n^2/2}$, $\sum q^{(n+\frac12)^2/2}$, and $\sum(-1)^nq^{n^2/2}$, respectively,
and are individually modular forms of weight~$1/2$ on the congruence subgroup~$\G(2)$ of~$\G_1$ but are 
permuted (up to roots of unity) by the action of~$\G_1$.  The space of modular forms of weight~$k=n/2$ spanned 
by $\th_2^n$, $\th_3^n$ and $\th_4^n$ is therefore $\MG$-invariant for any value of~$n$ and is 
3-dimensional for $n\ne4$.  It follows from the considerations given in Section~\ref{SolSpace} that this 
space is the solution space of an MLDE of order~3, and by looking at the first few coefficients of the 
$q$-expansions we find that this equation is $\d_k^3 f -\frac{3k^2-6k+8}{144}E_4\d_k f - \frac{k^2(k-6)}{864}f=0$,
which in terms of ordinary derivatives becomes 
\be\label{thetaNull} f''' \m \frac{k+2}4 E_2f'' \+ \Bigl(\frac{(k+1)(k+2)}4E_2'+ \frac k8E_4\Bigr)f' 
 \m \Bigl(\frac{k(k+1)(k+2)}{24} E_2''\+ \frac{k^2}{32}E_4'\Bigr)f \= 0\,. \ee

Finally, as already mentioned in the introduction, there are many examples of interesting MLDEs coming from the theory 
of vertex operator algebras and their characters.  Roughly speaking, for a wide class of vertex operator algebras there 
are finitely many irreducible modules, each of which has a character (a power series in~$q$ whose coefficients are
the dimensions of its graded pieces) that is known~\cite{Zhu} to be modular, and the space spanned by all of these
characters is always the solution space of some~MLDE. We omit all definitions here, since this is not our main subject,
but refer for instance to the \cite{MMS1, MMS2, AKNS,AN,ANSZ, KNS, MNS, RCM, NS1} and their bibliographies
for more details and examples.  A very simple case is the so-called (2,5) minimal model, whose two characters are
the famous Rogers-Ramanujan modular functions 
\be\label{RReqs}  G_0(\t) \= \sum_{n=0}^\infty\frac{q^{n^2-\frac1{60}}}{(1-q)(1-q^2)\cdots(1-q^n)}\,, \;\quad
G_1(\t) \= \sum_{n=0}^\infty\frac{q^{n^2+n+\frac{11}{60}}}{(1-q)(1-q^2)\cdots(1-q^n)}\,. \ee
These are individually only invariant under the congruence subgroup $\G(5)$ of~$\MG$ but
together span the solution space of the MLDE $\d^2f-\frac{11}{3600}E_4f=0$ (which can be written simply as 
$L_{2,1/5}f=0$ where $L_{2,k}$ is defined as in~\eqref{KZ}, but now with $k$ being $\frac15$ rather than an even 
integer)~\cite{MMS1}.
Similarly the third-order MLDO annihilating the $n$\thin th powers of the Thetanullwerte arises in connection with the 
space of characters of the lattice VOA associated to the root lattice~$D_n$. (See~\cite{MMS2} or~\S3 of~\cite{MNS}.)
Sometimes one can also use the modular transformation properties of solutions of an MLDE ``in reverse" to show the 
{\it non}-existence of VOAs of certain types.  For instance, the non-existence of a certain apparently possible 
simple VOA with central charge $c=164/5$ was proven in~\cite{AN} using the third order MLDE 
\be\label{L3}
  f''' \m \frac12\,E_2\,f'' \+ \Bigl(\frac12\,E_2'-\frac{169}{100}E_4\Bigr)\,f' \+ \frac{1271}{1080}\,E_6\,f\= 0\,, \ee
whose solution space is spanned by three explicit homogeneous polynomials of degree~82 in the Rogers-Ramanujan functions
$G_0$ and~$G_1$ with huge coefficients.  As a more complicated example, the 5th  order MLDE 
\ba \label{L5} & D^5(f) \m \frac{5}{3}E_2 D^4(f)+\Bigl(10E_2'+\frac{83}{99}E_4\Bigr) D^3(f)
 \m \Bigl(10E_2''+\frac{83}{66}E_4'+\frac{427}{3267}E_6\Bigr)D^2(f)\\
&\qquad \+ \Bigl(\frac{5}{3}E_2'''+\frac{83}{330}E_4''+\frac{427}{9801}E_6'+\frac{202}{107811}E_8\Bigr)D(f)
 \+\frac{7888}{39135393}E_{10}f\=0 \ea
is used in~\cite{NS1} in connection with the minimal model of minimal model of type (2,11) and central charge $c=-232/11$.

{\bf Remark.} Not all solutions of modular linear differential equations are necessarily modular forms.   For instance, for the 
Kaneko-Zagier equation,  one solution is always modular, and if the  parameter is~$\frac15$ there are two independent modular solutions
 (namely, the two Rogers-Ramanujan functions multipllied by a power of~$\eta(\t)$), but in general the second element of a basis
of solutions is a linear combination of products of modular forms and of Eichler integrals of modular forms of weight~2,
as discussed in~\cite{HK} and~\cite{Gu}.  This second type of solution was described in~\cite{Gu} as a mixed mock modular form,
and this is correct but is somewhat misleading since the mock modular forms occurring are of weight~0, so have
shadows of weight~2, and the non-holomorphic Eichler integral of a modular form of weight~2 is simply the
complex conjugate of the ordinary Eichler integral of a different modular form of weight~2, so that the notions
of ``mock" and ``mixed mock" are not needed at all.  (In fact, it seems probable that truly mixed mock modular
forms, i.e., sums of products of modular forms and of mock modular forms of weight different from zero, can
never be the solutions of any MLDE.)  Perhaps the best way to see these non-modular
solutions is as modular forms of second or higher order in the sense introduced in\cite{CDO} and~\cite{KlZ}, the latter in 
connection with a specific second order differential equation coming from percolation theory.

\section{Higher order examples: Rankin-Cohen brackets and higher Serre derivatives} \label{RCBandHSD}

The examples given above all had specified and quite small orders.  Here we consider three families of 
examples of MLDOs of arbitrary order.  The first and most important is given by the {\it Rankin-Cohen bracket}
\be\label{rc-bracket}
\bigl[f,\,g\bigr]^{(k,\,\ell)}_n\=\sum_{i=0}^{n}(-1)^i\binom{n+k-1}{n-i}\binom{n+\ell-1}{i}\,D^i(f) \,D^{n-i}(g)
   \qquad\bigl(n\in\Z_{\ge0}\bigr)\,,  
\ee
which belongs to $M_{k+\ell+2n}(\G)$ if $f\in M_k(\G)$ and $g\in M_\ell(\G)$ for any group~$\G$. 
These bilinear operations will play a key role in the whole paper and will be discussed and generalized 
in Section~\ref{ERCB}.  The map $f\mapsto[f,g]_n^{(k,\ell)}$ for fixed $g\in M_\ell(\G)$ and any $n\in\Z_{\ge0}$
is an MLDO of order~$n$ and type $(k,k+\ell+2n)$. As with the Serre derivative, we will often omit the superscripts
and write simply $[f,g]_n$ for $[f,g]_n^{(k,\ell)}$ when $f$ and $g$ are modular forms of weights~$k$ 
and~$\ell$.  

If $g=E_2$, then the Rankin-Cohen bracket $[\,\cdot\,,g]_n^{(\cdot,2)}$ no longer sends modular forms 
to modular forms, but it was discovered by Kaneko and Koike~\cite[p.~467]{KK} that the modified bracket 
\be \label{KKdef} \K_k^n(f) \: D^n(f) \m \frac{k+n-1}{12}\,\bigl[f,E_2\bigr]_{n-1}^{(k,2)} \ee
sends modular forms of weight $k$ to modular forms of weight $k+2n$ for arbitrary~$k$ and for all 
non-negative integers~$n$. This is an MLDO of type~$(k,k+2n)$ that reduces to the Serre derivative for $n=1$
and to the Kaneko-Zagier operator~$L_{2,k}$ for $n=2$ and that will play an important role in the sequel.  
We will give a proof of the modularity and an interpretation of $\K_k^n(f)$ for $f\in M_k$ as an ``extended 
Rankin-Cohen bracket" of~$f$ with a constant function in Section~\ref{ERCB}.  We should also mention that the
Kaneko-Koike has been rediscovered at least twice, by Henri Cohen and Frederik Str\"omberg (\!\!\cite{CS}, 
Prop.~5.3.27, p.~164) in the above form and by Xuanzhong Dai (\!\!\cite{Dai}, eq.~(1.7)) in the form of
a generalized Rankin-Cohen bracket (slightly different from ours) with~1.

Finally, we have the {\it canonical higher Serre derivatives} $\d_k^{[n]}$, which, like both the iterated Serre 
derivatives~$\d_k^n$ and the Kaneko-Koike operators~$\K_k^n$, are a family of monic MLDOs of order~$n$ and 
weight~$2n$ that reduce to the Serre derivative if~$n=1$, but which have much nicer properties than either 
of these other two families.  These operators were first used in\cite{VZ} in connection with the calculation of special 
values of $L$-functions of Hecke grossencharacters (whose definition again plays no role here and will be omitted) 
and are discussed in detail in Section~5 of the textbook~\cite{123}. They are defined recursively by
\be\label{defVZ}  \d_k^{[n+1]}(f) \= \d_{k+2n}\bigl(\d_k^{[n]}(f)\bigr)
  \m \frac{n(n+k-1)}{144}\,E_4\,\d_k^{[n-1]}(f) \;\qquad(n\ge1)  \ee    
with the initial values $\d_k^{[0]}(f)=f$, $\d_k^{[1]}(f)=\d_k(f)$, and have the attractive 
property that the Rankin-Cohen brackets are given by the {\it same} formula
\be\label{rc-bracket-vz}  \bigl[f, g\bigr]^{(k,\,\ell)}_n
 \=\sum_{i=0}^{n}(-1)^i\binom{n+k-1}{n-i}\binom{n+\ell-1}{i}\,\d^{[i]}(f) \,\d^{[n-i]}(g)\;
 \qquad\bigl(n\in\Z_{\ge0}\bigr) \ee
in terms of the higher Serre derivatives as in terms of the ordinary ones, but with the difference that now 
each individual term of the sum defining $[f,g]_n$ is modular and not just quasimodular.

As already mentioned, the higher Serre derivatives $\d_k^{[n]}$ actually form a much simpler family of 
higher order generalizations of~$\d_k$ than the more obvious iterated Serre derivatives $\d_k^n$, because 
they have a very simple expansion in terms of the usual derivatives~$D^r$ (equation~\eqref{defVZexpl} below), 
whereas no corresponding explicit form for the iterated derivatives $\d_k^n$ in terms of the $D^r$ is known.  
For the Kaneko-Koike operators $\K_k^n$, which can be seen as yet a third type of higher Serre derivatives,
the situation is intermediate, since they are given by a relatively explicit closed formula, 
but as linear combinations of the~$\d_k^{[n]}$ rather than of the $D^n$, as expressed in the following theorem.
\begin{theorem}\label{VZtoKK} The canonical higher Serre derivatives of a holomorphic function $f$ in~$\H$
are related to the usual derivatives by
\be\label{defVZexpl} \d_k^{[n]}(f) \= \sum_{r=0}^n\binom nr\,(k+r)_{n-r}\,
  \Bigl(-\frac{E_2}{12}\Bigr)^{n-r}\,D^r(f) \qquad(k,n\in\Z_{\ge0}), \ee
and the Kaneko-Koike derivatives are related to the canonical higher Serre derivatives by
\be\label{KKfromVZ}  \K_k^n(f) \= \sum_{m=0}^n\binom nm\,\binom{k+n-1}m\,\om_m\,\d_k^{[n-m]}(f) 
  \qquad(k,n\in\Z_{\ge0}), \ee
where the $\om_m$ are modular forms in $M_{2m}(\G_1)$ depending  only on~$m$, the first few being given by
\begin{center}  \def\arraystretch{1.5}  \begin{tabular}{c|ccccccccc}
$m$ & $\;0\;$ & $\;1\,$ & $\;2\;$ & $\;3\;$ & $\;4\;$ & $\;5\;$ & $\;6\;$ & $\;7\;$ & $\;8\;$ \\ \hline
$\om_m$ & $\;1\;$ & $0$ & $-\frac{E_4}{72}$ & $-\frac{E_6}{144}$  & $-\frac{E_4^2}{288}$ & $-\frac{5E_4E_6}{2592}$  
 & $-\frac{9E_4^3+16E_6^2}{20736}$ & $-\frac{35E_4^2E_6}{41472}$ & $-\frac{117E_4^4+128E_4E_6^2}{373248}$ \\
  \end{tabular} \;. \end{center} 
\end{theorem}
\noindent Here and from now on, $(x)_m$ denotes the ``shifted factorial" $x(x+1)\cdots(x+m-1)$.

The proof of Theorem~\ref{VZtoKK}, and a description of the forms~$\om_m$, will be given in Section~\ref{ERCB}.
\medskip

\section{Structure theorems}  \label{structure}

We now come to our central subject, the description of all MLDEs for a given lattice~$\G$.  We denote
by $\MLDO_{k,k+K}(\G)$ the space of all modular linear differential operators of \hbox{type~$(k,k+K)$} on~$\G$
and by  $\MLDO_{k,k+K}^{(\le n)}(\G)$ ($n\ge0$) the subspace of operators of order~$\le n$,
often omitting~$\G$ from the notations when it is equal to~$\G_1$. We write the expansion of an element 
$L\in\MLDO^{(\le n)}_{k,k+K}(\G)$ in the form
\be\label{a-exp} L\=\sum_{r=0}^n a_r(\t)\,D^r \ee
where $a_0,\dots,a_n$ are holomorphic functions in the upper half-plane. (At first sight it might seem more natural
to write an operator~$L$ of order~$n$ as $a_0D^n+\cdots+a_n$ rather than $a_0+\cdots+a_nD^n$, but this
numbering of the coefficients would not respect the filtration of $\MLDO_{k,k+K}$ or the addition of operators 
of different orders, and would also make our later formulas much more complicated.)  It is easy to see---and can
be seen clearly in the examples~\eqref{KZ}, \eqref{thetaNull}~\eqref{L3} or \eqref{L5} given in the last two
sectlions---that each $a_r$ must be a quasimodular form of weight~$K-2r$ and depth $\le n-r$ (where the 
{\it depth} of a quasimodular form will be discussed in Section~\ref{MLDOandQMF} for general lattices, but 
for $\G_1$ is just the degree of the form as a polynomial in $E_2$). In particular, since the weight of a 
holomorphic quasimodular form cannot be negative, we see that the order of any $L\in\MLDO_{k,k+K}(\G)$ is at 
most~$K/2$ and that we have an injective map 
\ba\label{QuasiModularSum}  
 \MLDO_{k,k+K}^{(\le n)}(\G) \quad & \mybig\hookrightarrow\quad\bigoplus_{r=0}^n \tM_{K-2r}^{(\le n-r)}(\G) \\
    a_0+a_1D+\cdots+a_nD^n \quad& \mybig\mapsto \quad (a_0,a_1,\dots,a_n)\,, \ea
where $\tM_\k^{(\le p)}(\G)$ denotes the space of quasimodular forms of weight $\k$ and depth~$\le p$ 
on the group~$\G$. In particular we have the dimension estimate
\be\label{DimEstimate} \dim\bigl(\MLDO_{k,k+K}^{(\le n)}(\G)\bigr) \;\,\le\;\,
 \sum_{r=0}^n\dim\bigl(\tM_{K-2r}^{(\le n-r)}(\G)\bigr)\;\,\le\;\, \sum_{r=0}^n\dim\bigl(\tM_{K-2r}(\G)\bigr) \ee
(independent of~$k$), but of course the actual dimension is much smaller, because the differential operator~$L$ 
defined by~\eqref{a-exp} with arbitrary quasimodular forms $a_r(\t)$ of weight~$K-2r$ and depth $\le n-r$ as 
coefficients would in general be only a {\it quasimodular} linear differential operator of weight~$K$, and in 
particular would send $M_k(\G)$ for any~$k$ to $\tM_{k+K}(\G)$, but not in general to $M_{k+K}(\G)$.  
There are therefore two natural questions:

\noindent{\bf Question 1.} What conditions must quasimodular forms~$a_r\in \tM_{k-2r}(\G)$ ($0\le r\le n$)
satisfy in order that the differential operator defined by~\eqref{a-exp} is modular?

\noindent{\bf Question 2.} What is the exact dimension of $\MLDO_{k,k+K}^{(\le n)}(\G)\,$?

\noindent The first question will be answered in two ways in Section~\ref{coeffs}. The second is easier and will be 
answered here, although for the moment only for $\G=\G_1$.

\begin{theorem}\label{Struct}
An MLDO of weight~$K$ on the full modular group has order at most $K/2$ and is a linear combination 
of Rankin-Cohen brackets with modular forms of positive weight, together with the $(K/2)$-nd Kaneko-Koike 
operator if the order is equal to~$K/2$.
\end{theorem}
\begin{corollary*} 
The dimension of $\MLDO_{k,k+K}^{(\le n)}=\MLDO_{k,k+K}^{(\le n)}(\G_1)$ is given for any $k$, $K$ and $n$ by
\be\label{DimMLDO} \dim\bigl(\MLDO_{k,k+K}^{(\le n)}\bigr) \= \sum_{r=0}^n\dim\bigl(M_{K-2r}\bigr)\,. \ee
\end{corollary*}
\begin{proof}  
We have already seen the first statement of the theorem, that the order of any $L$ in $\MLDO_{k,k+K}$
is at most $K/2$.  If it is exactly~$K/2$, then the top coefficient $a_n(\t)$ in~\eqref{a-exp} is
a (holomorphic) quasimodular form of weight~0 and hence constant.  Since the $n$\thin th Kaneko-Koike 
operator~\eqref{KKdef} begins with~$D^n$, we can simply subtract~$a_0\,\K_k^n$ from $L$ to reduce
it to an operator of lower order.  Now if the order is~$n<K/2$, then the top coefficient $a_n(\t)$ 
in~\eqref{a-exp} is a quasimodular form of weight $K-2n$ and depth~0, and hence actually a modular form.
Then, since the Rankin-Cohen bracket~\eqref{rc-bracket} of a modular form $f$ of weight~$k$ and a modular
form $g$ of weight~$\ell>0$ has an expansion beginning with a non-zero multiple of $f^{(n)}g$, we see
that by subtracting from $L$ a suitable non-zero multiple of $\bigl[\,\cdot\,,a_n\bigr]_n^{(k,K-2n)}$
we can again reduce its order by at least one.  Continuing by induction, we obtain the theorem. The corollary 
is then immediate since $\dim M_0=1$ and since the $(K/2)$\thin th Kaneko-Koike operator and the $r$\thin th 
Rankin-Cohen brackets with modular forms of strictly positive weight $K-2r$ are easily seen to be linearly independent.
\end{proof}

Theorem~\ref{Struct} gives us canonical isomorphisms 
\be\label{ModularSum}  \MLDO_{k,k+K} \; \cong \;\bigoplus_{r\ge0} M_{K-2r}\,, \qquad
  \MLDO_{k,k+K}^{(\le n)} \; \cong \;\bigoplus_{r=0}^n M_{K-2r}\,, \ee
defined (from right to left) by mapping $1\in M_0$ to the Kaneko-Koike operator~\eqref{KKdef} (or to any chosen 
multiple of it) and $g\in M_{K-2r}$ with $K-2r>0$ to the $r$\thin th Rankin-Cohen bracket $[\,\cdot\,,g]_r^{k,K-2r}$ (or again
to any chosen multiple of it, where the multiples in each case can be arbitrary non-zero numbers depending only on $k$, $K$ and~$r$).
But three further descriptions can be obtained by noticing that any MLDO, as well as the expansion~\eqref{a-exp} in 
terms of ordinary higher derivatives, has three further expansions
\be\label{bcd-exp} L \=\sum_{r=0}^n b_r(\t)\,\d_k^r  \=\sum_{r=0}^n c_r(\t)\,\K_k^r  \=\sum_{r=0}^n d_r(\t)\,\d_k^{[r]} \ee
in terms of the three different types of higher-order Serre derivatives introduced in Section~\ref{RCBandHSD}, and unlike
the situation for the original expansion~\eqref{a-exp} the condition on the coefficients here is trivial: 
since each of $\d_k^r$, $\K_k^r$, and $\d_k^{[r]}$ preserves modularity and increases the weight by~$2r$, we simply need that
each of the three coefficients $b_r$, $c_r$ and $d_r$ belongs to $M_{K-2r}$ for each~$r$.  This gives
  
\begin{theorem}\label{bcd-Struct} We have three isomorphisms~\eqref{ModularSum} given by mapping 
an operator $L\in \MLDO^{(\le n)}_{k,k+K}$ with the three expansions~\eqref{bcd-exp} to one of the vectors 
$(b_0,\dots,b_n)$, $(c_0,\dots,c_n)$, or $(d_0,\dots,d_n)$.
\end{theorem}

Thus, although the final statement of the isomorphism of $\MLDO_{k,k+K}$ or $\MLDO^{(\le n)}_{k,k+K}$ to
a known space is the same in Theorem~\ref{bcd-Struct} as in Theorem~\ref{Struct}, the actual isomorphisms
are completely different: whereas  in Theorem~\ref{Struct} we mapped a modular form $g\in M_{K-2r}$
to (a multiple of) $[\,\cdot\,,g]_r$ or $\K_k^r(\,\cdot\,)$ (depending whether $2r<K$ or $2r=K$), which is a 
linear combination of derivatives of~$g$ times powers of~$D$, we now map the same~$g$ to an operator
that again begins with $gD^r$ but is now simply a multiple of~$g$, no longer containing any of its higher derivatives.

We can summarize the whole discussion of this section formally in terms of the filtration and its splittings.  
For any lattice $\G\subset\SL$ and any integer $n\ge0$ we have a canonical map  $\s$ from 
$\MLDO^{(\le n)}_{k,k+K}(\G)$ to $M_{K-2n}(\G)$ assigning to an operator $L$ with the $D$-expansion~\eqref{a-exp} its
top coefficient~$a_n$, which we call the {\it symbol} of~$L$.  For $\G=\G_1$ this gives a short exact sequence 
\be\label{exact}  0 \,\longrightarrow\, \MLDO^{(\le n-1)}_{k,k+K} \,\longrightarrow\, 
      \MLDO^{(\le n)}_{k,k+K} \,\overset\s\longrightarrow\, M_{K-2n} \,\longrightarrow\, 0\,, \ee
where the exactness at all places except the last (i.e., the surjectivity of~$\s$) is clear and the last follows
there from any of the four splittings of~$\s$ that we have given (mapping $g\in M_{K-2n}$ to the 
$n$\thin th Rankin-Cohen bracket with $g$ if $K-2n>0$ and to the $n$\thin th Kaneko-Koike operator \hbox{if~$g=1$}, 
or else multiplying it with any of the three Serre derivatives of order~$n$ defined in Section~\ref{RCBandHSD}),
corresponding to the four isomorphisms between MLDOs and vectors of modular forms described in Theorems~\ref{Struct}
and~\ref{bcd-Struct}. The space $\MLDO_{k,k+K}$ is filtered by the order and the exact sequence~\eqref{exact} gives 
a canonical identification of the graded vector space $\bigoplus_{n\ge0}\MLDO^{(\le n)}_{k,k+K}(\G)/\MLDO^{(\le n-1)}_{k,k+K}(\G)$ 
 associated to this filtration with~$\bigoplus_{n\ge0} M_{K-2n}$,
with each of the four splittings giving an isomorphism of $\MLDO_{k,k+K}$ with the same space.
Since all four isomorphisms induce the {\it same} isomorphism of graded as opposed to filtered 
vector spaces (because the exact sequence~\eqref{exact}, unlike its splittings, is canonical), the passage from any one to any other
can be described in terms of block unipotent matrices (blocks of size $\dim M_{K-2r}$, with zeros
below the diagonal and identity matrices on the diagonal), but not all of these can be made explicit.
In fact, the only passage between the different isomorphisms that we know how to describe in closed form
is that between the isomorphism induced by the Kaneko-Koike operators $\K_k^r$ and the canonical higher
Serre derivatives $\d_k^{[r]}$, because of Theorem~\ref{VZtoKK}.

We end this section by giving examples of all four isomorphisms we have described for some of the explicit
examples of MLDOs introduced in Section~\ref{examples}. For the Kaneko-Zagier operator $L_{2,k}$ this is almost
trivial: it was already given as a combination of $D^r$ and of $\d_r$ in equations~\eqref{KZ} and~\eqref{KZDO}; 
its expression as a linear combination of $\K_k^2$ and a Rankin-Cohen bracket or as a linear combination with
modular coefficients of $\K_k^r$ is simply $\K_k^2$ itself, and in terms of the higher Serre derivatives 
it equals $\d_k^{[2]} -\frac{k(k+1)}{144}E_4$.  For the third-order MLDO $L_{3,k}^\theta$ of type $(k,k+6)$
defined by the left-hand side of~\eqref{thetaNull}, the expression corresponding to Theorem~\ref{Struct} is 
$$ L_{3,k}^\theta f \= \K_k^3(f) \m \frac k{32}\, \bigl[f,E_4\bigr]^{(k,4)}_1  $$
and those corresponding to Theorem~\ref{bcd-Struct} are
\ba\label{alpha}
 L_{3,k}^\theta f & \= \d_k^3(f) \m \frac{3k^2-6k+8}{144}E_4\d_k(f)  \m \frac{k^2(k-6)}{864}\,E_6\,f \\
   &\= \K_k^3(f) \+ \frac k8\,E_4\,\K_k^1(f)  \+ \frac{k^2}{96}\,E_6\,f \\
  & \= \d_k^{[3]}(f) \m \frac{(k-1)(k-2)}{48}\,E_4\,\d_k(f) \m \frac{k(k^2-6k+2)}{864}\,E_6\,f\,,
\ea
while for the third-order MLDO of type (0,6) defined by $L_3=D^3-\frac12E_2D^2 +(\frac12E_2'-\frac{169}{100}E_4)D$
(i.e.,~the left-hand side of equation~\eqref{L3} without the final term, which is modular anyway) the 
corresponding expressions, which are easier here because $k$ is~0 rather than a variable, are
\ba\label{beta}
L_3(f) &\= \K_0^3(f) \+ \frac{169}{400}\,\bigl[f,E_4\bigr]^{(k,4)}_1 \\ 
 & \= \d_0^3(f) \m \frac{1571}{900}\,E_4\,\d_0(f)  
   \= \K_0^3(f) \m \frac{169}{100}\,E_4\,\K_0^1(f)   \= \d_0^{[3]}(f) \m \frac{1039}{600}\,E_4\,\d_0(f)  \,.
\ea
\smallskip
Similarly, writing $L_5$ for the operator in~\eqref{L5} without its last term, we find the four representations
\ba\label{gamma}
L_5(f) &\= \K_0^5(f) \m  \frac{83}{1980}\,\bigl[f,E_4\bigr]^{(0,4)}_3 \m 
   \frac{61}{9801}\,\bigl[f,E_6\bigr]^{(0,6)}_2 \m \frac{101}{431244}\,\bigl[f,E_8\bigr]^{(0,8)}_1   \\
&\=\d_0^5(f) \m \frac{53}{396}\,E_4\,\d_0^3(f) \+ \frac{295}{8712}\,E_6\,\d_0^2(f) \m \frac{6151}{1724976}\,E_8\,\d_0(f) \\
&\=\K_0^5(f) \+ \frac{83}{99}\,E_4\,\K_0^3(f) \+ \frac{1885}{6534}\,E_6\,\K_0^2(f) \+ \frac{7181}{143748}\,E_8\,\K_0(f) \\
&\=\d_0^{[5]}(f) \+ \frac1{198}\,E_4\,\d_0^{[3]}(f) \+ \frac{35}{3267}\,E_6\,\d_0^{[2]}(f) 
  \m \frac{2689}{1149984}\,E_8\,\d_0(f)\;.
\ea

\section{The expansion coefficients of modular linear differential operators}  \label{coeffs}

We now turn to the first question posed in the last section, namely, the determination of the conditions 
that must be satisfied by the coefficients $a_r$ in~\eqref{a-exp} in order that the differential operator 
defined there is modular.

Of course, from one point of view the answer to this question is almost tautological: we define
a different series of coefficients $b_s$ by requiring the first expansion in~\eqref{bcd-exp} to hold; these 
are computable combinations of the $a$'s and their derivatives and hence are automatically quasimodular of
the correct weights, and $L$ is an MLDO if and only if they are actually modular.  For small orders~$n$
we can carry this process out by hand, but for general~$n$ we cannot, because the expansion of $\d_k^s$
as a polynomial in $D$ with quasimodular coefficients is not known in closed form. We can, in fact, do
this if we use the third expansion in~\eqref{bcd-exp} together with equation~\eqref{defVZexpl} and its easy
inversion, in which case we find the following result, whose proof will be given in Section~\ref{ERCB}
when we prove~\eqref{defVZexpl}:
\begin{theorem}\label{atob} The operator $L$ defined by~\eqref{a-exp} is modular of type $(k,k+K)$ if and only if the function
\be\label{a-b}  d_r(\t) \: \sum_{j\ge 0} \binom{r+j}j\,(k+r+j)_j\,\Bigl(\frac{E_2(\t)}{12}\Bigr)^j\,a_{r+j}(\t) \ee
is a modular form for $r=0,\dots,n$, in which case $L$ is given by the final expression in~\eqref{bcd-exp}.
\end{theorem}

It is nevertheless interesting to give a direct description of the conditions that the original 
coefficients~$a_r$ must fulfill in order for~$L$ to be modular.  We will do this first in terms of the
modular transformation properties that they must satisfy and then give a second characterization
by constructing certain explicit linear combinations of the~$a_r$ and their derivatives whose 
modularity is a necessary and sufficient condition for that of the operator~$L$. 

To study the transformation properties of the $a_r$, we will need the following lemma.

\begin{lemma} \label{Transf1}
For any holomorphic function $f$, $\g=\sm abcd\in\SL$, $k\in\Z$ and $m\ge0$ we have
\be\label{transf1} D^m\bigl(\bigl(f\sl k\g\bigr)(\t)\bigr) \= 
\sum_{r=0}^m\binom mr \,(k+r)_{m-r}\,\Bigl(-\frac 1{2\pi i}\,\frac c{c\t+d}\Bigr)^{m-r}\,\bigl((D^rf)\sl{k+2r}\g\bigr)(\t) \;. \ee
\end{lemma}
\begin{proof}
This follows by induction on~$m$, using $D(\g\t)=\frac1{2\pi i}\,\frac1{(c\t+d)^2}$ and the elementary
identity $ (k+m+r)\binom mr(k+r)_{m-r}+\binom m{r-1}(k+r-1)_{m-r+1}=\binom{m+1}r(k+r)_{m+1-r}$.
\end{proof}

\noindent For completeness and for later use, we also give the counterpart to this lemma in the 
other direction (first differentiate, then apply a modular transformation rather than vice versa), namely
\be\label{transf2} \bigl((D^mf)|_{k+2m}\g\bigr)(\t) \= 
   \sum_{r=0}^m\binom mr \,(k+r)_{m-r}\,\Bigl(\frac1{2\pi i}\,\frac c{c\t+d}\Bigr)^{m-r}\,\bigl(D^r(f|_k\g)\bigr)(\t)\;. \ee
The proof of this formula (which is also stated in~\cite{123}, top of page~54, in the special case when $f\in M_k(\G)$)
is similar to that of~\eqref{transf1} and will be omitted here (but indicated briefly in Section~\ref{ERCB}). 

Using the lemma we can easily determine the modular transformation property of the vector of coefficients~$a_r(\t)$
in~\eqref{a-exp} needed to make the operator~$L$ modular.

\begin{theorem}\label{a-transf} Let $a_r(\t)$ $(0\le r\le n)$ be holomorphic functions
in the upper half-plane.  Then the differential operator~$L$ defined by~\eqref{a-exp} is an MLDO of 
type $(k,k+K)$ with respect to a lattice~$\G$ if and only if the $a_r(\t)$ transform by
\be \label{a-transfeq}  \bigl(a_r\sl{K-2r}\g\bigr)(\t)
 \= \sum_{j=0}^{n-r} \binom {r+j}j \,(k+r)_j\,\Bigl(-\frac 1{2\pi i}\,\frac c{c\t+d}\Bigr)^j\,a_{r+j}(\t) \ee
for all $0\le r\le n$, all $\t\in\H$, and all $\g=\sm abcd\in\G$.
\end{theorem}
\begin{proof}  Multiplying both sides of~\eqref{transf1} by~$a_m(\t)$ and summing over~$m$, we find
$$ L\bigl(f\sl k\g)(\t) \= \sum_{0\le r\le m\le n} a_m(\t) \,
    \binom mr \,(k+r)_{m-r}\,\Bigl(-\frac 1{2\pi i}\,\frac c{c\t+d}\Bigr)^{m-r}\,\bigl((D^rf)\sl{k+2r}\g\bigr)(\t)\;. $$
On the other hand, from the definition of the slash operation it follows that
$$ \bigl(L(f)|_{k+K}\g\bigr)(\t) \= \sum_{r=0}^n \bigl(a_r\sl{K-2r}\g\bigr)(\t)\,\bigl(D^r(f)\sl{k+2r}\g\bigr)(\t)\;. $$
Comparing the coefficients of $D^r(f)|_{k+2r}\g$, we see that~\eqref{MLDEdef} is equivalent to~\eqref{a-transfeq}.
\end{proof}

Theorem~\ref{a-transf} gives a complete answer to the first question posed in the previous section, since
it gives a necessary and sufficient criterion for the coefficients in~\eqref{a-exp} to make $L$ modular,
but it is not a very easy one to check in practice.  We would prefer to replace this description by one
in terms of actual modular forms, in which case no further transformation properties are needed.
To do this, we look at the first few cases of the transformation equation~\eqref{a-transfeq}.  If $r=n$, the top coefficient, it says
\be\label{case0}  (a_n\sl{K-2n}\g)(\t) \= a_n(\t) \ee
for all $\g\in\G$, i.e.,~that $a_n$ is a modular form of weight~$K-2n$, which of course we knew already.  
For the next case $r=n-1$ it says
\be\label{case1}  (a_{n-1}\sl{K-2n+2}\g)(\t) \= a_{n-1}(\t) \m n(k+n-1)\,\frac1{2\pi i}\,\frac c{c\t+d}\,a_n(\t)\,. \ee
But differentiating~\eqref{case0} with respect to~$\t$ gives
\be\label{case0d}   (a'_n\sl{K-2n+2}\g)(\t) \= a'_n(\t) \+ (K-2n)\,\frac1{2\pi i}\,\frac c{c\t+d}\,a_n(\t)\,, \ee
and taking a suitable linear combination of this and~\eqref{case1} we find that the linear combination
\be\label{Badh} h_{n-1}(\t) \: a_{n-1}(\t) \+ \frac{n(k+n-1)}{K-2n}\,a_n'(\t) \ee
transforms like a modular form of weight~$K-2n+2$.  Similarly, for the next case we find that
$$h_{n-2}(\t) \: a_{n-2}(\t) \+ \frac{(n-1)(k+n-2)}{K-2n+2}\,a_{n-1}'(\t) 
  \+ \frac{n(n-1)(k+n-2)(k+n-1)}{2(K-2n+1)(K-2n+2)}\,a_n''(\t) $$
is modular of weight $K-2n+4$, and looking also at the next case we easily guess that the general statement 
should be that the function
\be\label{Defhm} h_m(\t) \: \sum_{s\ge0}\binom{m+s}s\,\frac{(k+m)_s}{(K-2m-s-1)_s}\;a^{(s)}_{m+s}(\t) \ee
(where the sum stops at $s=n-m$ because the later $a_{m+s}$ vanish) is modular of weight $K-2m$ for 
$0\le m\le n$.  Since we know that $K\ge2n$, we see that the denominators in each of
these formulas are non-zero except in the case of~\eqref{Badh} when $K=2n$, in which case $a_n$ is
a constant and the second term becomes 0/0.  We will return to this point in a moment, but for now
will simply assume that the order of our differential operator is strictly less than~$K/2$.

\begin{theorem}\label{a-h} Let $a_r(\t)$ $(0\le r\le n)$ be holomorphic functions
in the upper half-plane.  Then the differential operator~$L$ defined by~\eqref{a-exp}, where $n<K/2$, is an MLDO of 
type $(k,k+K)$ with respect to a lattice~$\G$ if and only if the function $h_m$ defined by~\eqref{Defhm}
belongs to $M_{K-2m}(\G)$ for all $0\le m\le n$. If this is the case, then the expansion of~$L$ as a linear
combination of Rankin-Cohen brackets is given by
\be\label{RCBexplicit}  L(f) \= \sum_{m=0}^n {\binom{K-m-1}m}\i\,\bigl[h_m,f\bigr]_m^{(K-2m,k)}\;. \ee
\end{theorem}
\begin{proof} 
If $L$ is modular, then by induction on~$s$, starting with Theorem~\ref{a-transf} for~$s=0$, we find
$$ a_r^{(s)}\sl{K-2r+2s}\g \= \sum_{\substack{j\ge0\\ 0\le t\le s\substack}} \binom {r+j}j\binom st 
\,(k+r)_j\,(j-K+2r-s+1)_t\,U^{j+t}\,a_{r+j}^{(s-t)}(\t)  $$
for all $s\ge0$, where we have set $U=U_\g(\t)=-\frac 1{2\pi i}\frac c{c\t+d}$ and omitted the variable $\tau$ for
simplicity.  Inserting this into~\eqref{Defhm} gives
$$ h_m\sl{K-2m}\g \=  \sum_{\ell,p\ge0} \frac{(m+p+\ell)!}{m!\,p!\,\ell!}\,\frac{(k+m)_{p+\ell}\;U^\ell}{(K-2m-p-1)_p}
\,a^{(p)}_{m+\ell+p}\,\sum_{t=0}^\ell (-1)^t\,\binom\ell t \frac{(K-2m-t-p-\ell)_t}{(K-2m-t-p-1)_t} $$
and this equals $h_m$ because the interior sum vanishes for~$\ell>0$, since it is the $\ell$\thin th difference of a 
polynomial of degree~$\ell-1$.  This proves the first assertion.  For the second, we note that the inversion 
of~\eqref{Defhm}, expressing the $a$'s as linear combinations of the derivatives of the $h$'s, is given by
\be\label{inversion}  a_r \= \sum_{j\ge0} (-1)^j\,\binom{r+j}j\,\frac{(k+r)_j}{(K-2r-2j)_j}\, h_{r+j}^{(j)} \,, \ee
To see this, insert~\eqref{Defhm} into~\eqref{inversion} and set $X=K-2r-1$ and $r+j=m$ to find
$$ \text{RHS of~\eqref{inversion}} \= \sum_{m\ge0}\frac{(r+1)_m\,(k+r)_m}{(X-2m)_{m+1}}\,a_{r+m}^{(m)}\,
   \sum_{j=0}^m\,(X-2j)\,\binom Xj\binom{2m-X}{m-j}\,, $$
which equals $a_r$ because the inner sum is equal to $X\bigl[\binom{2m}m-2\binom{2m-1}{m-1}\bigr]=0$ for $m\ge1$.
Now insert equation~\eqref{inversion} into~\eqref{a-exp} and use the definition of the Rankin-Cohen
brackets to obtain~\eqref{RCBexplicit}.
\end{proof}

We now consider the case $K=2n$ omitted so far, assuming that~$\G=\G_1$.  (General lattices will be
treated in Section~\ref{cocompact}.)  In that case the function $a_n$ is a non-zero constant, which
we can assume without loss of generality to be~1, so that $L=D^n+$(lower order terms) is monic.  Recall that 
the problem here only concerned the function $h_{n-1}$, which cannot be defined by~\eqref{Badh}.  But~\eqref{case1} 
with $a_0=1$ now reduces to $(a_{n-1}\sl2\g)(\t)=a_{n-1}(\t) - \frac{n(k+n-1)}{2\pi i}\,\frac c{c\t+d}$,
and comparing this with the transformation law of the weight~2 quasimodular form $E_2$, which will be recalled
in the next section, we see that $a_{n-1}-\frac{n(k+n-1)}{12}E_2$ is a modular form of weight~2 on~$\G_1$ and 
hence is equal to~0.  A computation similar to that of Theorem~\ref{a-h}, though slightly more complicated, then 
gives the following theorem, which together with Theorem~\ref{a-h} gives the desired explicit version of 
Theorem~\ref{Struct}.  We omit the somewhat messy direct combinatorial proof of this theorem since we will find smoother
approaches in Section~\ref{MLDOandQMF} using quasimodular forms and in Section~\ref{ERCB} using extended Rankin-Cohen brackets.

\begin{theorem}\label{Monic} Let $L$ be a monic MLDO of order~$n$ and type $(k,k+2n)$ on the full modular group.  
Then the function~$h_m$ defined by~\eqref{Defhm} is modular of weight $2n-2m$ for $0\le m\le n-2$ and
\be\label{Monicexplicit}  L(f) \= \K_k^n(f) \+ \sum_{m=0}^{n-2} {\binom{2n-m-1}m}\i\,\bigl[h_m,f\bigr]_m^{(K-2m,k)}\;. \ee
\end{theorem}

\section{Quasimodular forms and modular linear differential operators} \label{MLDOandQMF}

In Section~\ref{structure} we gave four different descriptions of MLDOs in terms of modular forms,
using Rankin-Cohen brackets and higher-order Serre derivatives, the final result in each case
being an isomorphism of filtered vector spaces as in~\eqref{ModularSum}. 
However, there is another space, much more familiar than the space of MLDOs, that is also canonically isomorphic
(this time in two rather than four different canonical ways) to the same direct sum of spaces of modular forms,
 namely the space of all {\it quasimodular} forms of weight~$K$ and depth~$\le n$.  In this section we show that there is
a direct and extremely simple correspondence between quasimodular forms and MLDOs for any lattice, and that combining 
this correspondence with the two isomorphisms between quasimodular forms and vectors of ordinary modular forms 
gives two of the four descriptions of MLDOs given in Section~\ref{structure} for~$\G=\G_1$.  We will postpone to
Section~\ref{cocompact} the discussion of the extent to which these latter results hold for other lattices
(roughly speaking, in the cocompact case there is only one natural description of quasimodular forms in terms 
of modular forms and only one of the four isomorphisms in question generalizes, but in the non-cocompact case 
they all go through if one modifies the definitions appropriately), but we will give the definition and properties 
of quasimodular forms for arbitrary~$\G$ already here.

The key structure on quasimodular forms is that they form a module over the Lie algebra~$\lie$.
More explicitly, for any lattice $\G$ there are three natural derivations $D$, $W$, and~$\de$
on the space $\tM_*(\G)$ of quasimodular forms on~$\G$ that satisfy the commutation relations 
\be \label{commutators} [W,D] \= 2D,\qquad [W,\de] \= -2\de, \qquad [\de,D]\= W\;. \ee
Here $D$ is the renormalized differentiation operator $(2\pi i)^{-1} d/d\t$ as above and $W$ is
the weight operator multiplying an element of $\tM_k(\G)$ by~$k$ (so that the first two 
commutation relations in~\eqref{commutators} just say that $D$ and $\de$ increase and decrease 
the weight by~2, respectively).  The definition of the operator $\de$ is more complicated.
For the case of the full modular group it is simply $12\,\p/\p E_2$ if we identify $\tM_*(\G_1)$ with 
$\C[E_2,E_4,E_6]$, the commutation relations~\eqref{commutators} then being direct consequences of the 
Ramanujan formulas~\eqref{RamDer}. To define it for general lattices, we must first review the definition 
and basic properties of quasimodular forms.  We refer to~\S5.3 of~\cite{123} (where $W$ is denoted~$E$)
for more details and for all proofs omitted here, none of which are difficult.

In fact there are two different definitions of quasimodular forms.  The one used in~\cite{KZ}, where 
the term was first introduced, was in terms of almost holomorphic modular forms and will be reviewed
in the last section of this paper.  The other, which was suggested subsequently by Werner Nahm, is
more direct and more algebraic and will be discussed here.

The starting point for this definition is the observation that if we differentiate the transformation 
law $f\bigl(\frac{a\t+b}{c\t+d}\bigr)=(c\t+d)^kf(\t)$ of a modular form of weight~$k$ on some 
group~$\G\subset\SL$, we find after multiplying through by $(c\t+d)^2$ that
\be \label{dftransf} f'\Bigl(\frac{a\t+b}{c\t+d}\Bigr) \= (c\t+d)^{k+2}f'(\t) \+ \frac k{2\pi i}\,c(c\t+d)^{k+1}f(\t)\
  \qquad \text{for $\bm abcd\in\G$}\ee
(recall that $Df=(2\pi i)^{-1}\,df/d\t$), which describes precisely the extent to which $f'$ fails to be a 
modular form~of weight~$k+2$ on~$\G$.  (We already used this identity in the previous section when passing 
from equation~\eqref{case0} to equation~\eqref{case0d}.) Similarly, the weight~2 Eisenstein series defined 
in~\eqref{EisDef} (with $C_2=-\frac1{24}$), which is quasimodular of weight~2 and depth~1, satisfies (cf.~\cite{Ser})
\be\label{E2transf}
E_2\Bigl(\frac{a\t+b}{c\t+d}\Bigr) \= (c\t+d)^2E_2(\t)+\frac{6}{\pi i}\,c(c\t+d)\qquad\text{for $\bm abcd\in \MG$}\;. \ee
We can rewrite~\eqref{dftransf} and~\eqref{E2transf} using the slash operator~\eqref{Slash} as
$ \bigl(f'\bigr|_{k+2}\g\bigr)(\t)=f'(\t)+\frac{kf(\t)}{2\pi i}\,\frac c{c\t+d}$ and 
$(E_2\sl2\g)(\t)=E_2(\t)\+\frac6{\pi i}\frac c{c\t+d}$, respectively. These two examples motivate the following
general definition: A {\it quasimodular form} of weight~$K$ on a lattice~$\G\subset\SL$ is a holomorphic function 
$F:\Hh\to\C$ of moderate growth such that the map $\g\mapsto(F|_K\g)(\t)$ from $\G$ to $\C$ is a polynomial 
in $\frac c{c\t+d}$ (where $\g=\sm abcd$) for any fixed value of~$\t\in\Hh$.  If $n$ is the maximum degree 
of these polynomials for all $\t\in\Hh$, we say that $F$ is of {\it depth}~$n$.  Thus a quasimodular form of 
weight~$K$ and depth~$n$ has the transformation behavior 
\be\label{qmftransf} \bigl(F|_K\g\bigr)(\t)\= \sum_{r=0}^n F_r(\t)\,\Bigl(\frac1{2\pi i}\,\frac c{c\t+d}\Bigr)^r \ee
for all $\g=\sm abcd\in\G$, where $F_0,\dots,F_n$ are holomorphic functions (of moderate growth, as usual)
in the upper half-plane that are independent of~$\g$.  Taking $\g$ to be the identity, we see that $F_0=F$,
and we can now define $\de$ by setting $\de(F)=F_1$.  It is then not very difficult to prove that $F_r=\de^rF/r!$
for all~$r\ge0$ , so
that~\eqref{qmftransf} takes the form of a ``modified Taylor expansion"
\be\label{qmfTaylor} \bigl(F|_K\g\bigr)(\t)
  \= \sum_{r=0}^\infty \frac{\de^r\!F(\t)}{r!}\,\Bigl(\frac1{2\pi i}\,\frac c{c\t+d}\Bigr)^r\,. \ee
Note that the infinite series always terminates, because $\de^rF$ has weight $K-2r$ and hence vanishes for $2r>K$.
The depth of~$F$ is by definition the largest value of~$r$ for which $\de^r\!F$ is non-zero. 

We can now combine these ideas with the results in the previous section to obtain a description of
all MLDOs on an arbitrary lattice in terms of quasimodular forms.  Specifically, by
comparing~\eqref{qmfTaylor} with equation~\eqref{a-transfeq} (Theorem~\ref{a-transf}), we find immediately
that each coefficient $a_r$ ($0\le r\le n$) of~$L$ in~\eqref{a-exp} is quasimodular of weight~$K-2r$ and
depth~$n-r$, as mentioned earlier, and that the action of powers of~$\de$ on this coefficient is given by
\be\label{actionofd}  \de^ja_r \= (-1)^j\,(r+1)_j\,(k+r)_j\,a_{r+j} \ee
for all~$j\ge0$. These equations are actually overdetermined, since already the case $j=1$ tells us
that $\de a_{r-1}=-r(k+r-1)a_r$ and hence by induction on~$r$ that each $a_r$ has the form $(-1)^r\de^ra_0/r!(k)_r$,
which is what we would get directly from~\eqref{actionofd} with~$r=0$ and which in turn implies~\eqref{actionofd} 
for all $r$ and~$j$.  This establishes the following bijection between MLDOs and quasimodular forms:
\begin{theorem}\label{QMFtoMLDO} Let $F$ be a quasimodular form of weight~$K$ and depth~$n$ on
an arbitrary lattice~$\G$.  Then for every positive integer $k$ the operator 
\be\label{pairing}  L_{F,k} \: \sum_{r=0}^n \frac{(-1)^r}{r!\,(k)_r}\,\de^r(F)\thin D^r\, \ee
is a modular linear differential operator of order~$n$ and type~$(k,k+K)$. For fixed~$k>0$,
the map $F\mapsto L_{F,k}$ gives an isomorphism
\be\label{isomQMtoMLDO} 
 \tM_*(\G)\; \xrightarrow{\sim}\; \MLDO_{k,k+*}(\G)
\ee
of filtered graded $M_*(\G)$-modules, where $\qM_*(\G)$ is filtered by depth and $\MLDO_{k,k+*}(\G)$ by order, 
with the inverse isomorphism mapping the differential operator~\eqref{a-exp} to its ``constant term"~$a_0(\t)$.
\end{theorem}

\noindent As a simple example of the first statement of the theorem, if we take for $F$ the quasimodular
form $(-\frac1{12}E_2)^n$ of weight~$2n$ and depth~$n$, then $L_{F,k}(f)$ equals $\d^{[n]}(f)/(k)_n$ with 
$\d^{[n]}(f)$ as in equation~\eqref{defVZexpl}.  We should also remark that the weaker statement that 
$L_{F,k}(f)$ is a modular form of weight~$k+K$ for any $F\in\tM_K(\G)$ and $f\in M_k(\G)$ can be seen 
without the rather complicated calculations of Section~\ref{coeffs} simply from the $\lie$-commutation 
relations~\eqref{commutators}, since if $f$ is a modular form of weight~$k$ then it satisfies
$W(f)=kf$ and $\de(f)=0$ and then by a standard and easy induction also $\de(D^rf)=r(k+r-1)D^{r-1}f$ 
(cf.~equation~\eqref{GSM25} below) and therefore $\de(L_{F,k}(f))=0$. The last statement of the theorem,
which asserts that $L_{Fh,k}=hL_{f,k}$ for any modular form~$h$, is true because~$\de^r(Fh)=\de^r(F)\thin h$ since
$\de$ is a derivation and annihilates~$h$.

Theorem~\ref{QMFtoMLDO} makes sense and remains true if $k$ is negative and non-integral, but if~$k$ is~0 
or a negative integer then the right-hand side becomes meaningless.  We will see one way to remedy this
in the case~$k=0$ in Section~\ref{ERCB2} by identifying $L_{F,k}(f)$ for~$f\in M_k$ with the
value $\{f,F\}$ of a pairing between modular and non-modular forms that makes sense also when~$k=0$ and~$f=1$.
However, to understand MLDOs we do not merely want their action on modular forms, but on arbitrary differentiable
functions, and here this is quite easy to do just by rescaling.  For instance, if we multiply~\eqref{pairing}
by~$k$ and set $k=0$, we get the operator $\sum_{r\ge1}\frac{(-1)^r}{r!(r-1)!}\de^r(F)\thin D^r$, which is indeed
an MLDO of type~$(0,K)$ (equal, for example, to~$D$ if~$F=-E_2/12$, $K=2$).  More generally, if we divide~\eqref{pairing}
by~$(k-1)!$, then this does not change any of the statements of the theorem as long as $k$ is a strictly positive
integer, but gives a new operator $\sum \frac{(-1)^r}{r!(k+r-1)!}\thin \de^r(F)\thin D^r$ that is well-defined
also for $k\in\Z_{\ge0}$ if we interpret the terms having the factorial of a strictly negative integer in their 
denominator as~0.  Then for~$\G$ non-cocompact all of the statements of the theorem except for the statement
about the inverse isomorphism remain true (for instance, for~$\G=\G_1$, $k=0$ and~$K=2$ the spaces
$\MLDO_{0,2}(\G)=\C\cdot D$ and $\qM_2(\G)=\C\cdot E_2$ are isomorphic, but we cannot see this by taking
the ``constant term"), whereas for co-compact lattices (which will be treated in Section~\ref{cocompact}) 
the map~\eqref{pairing} still exists but is no longer always an isomorphism. 

Theorem~\ref{QMFtoMLDO} gives a more conceptual explanation of the expansions discussed in Section~\ref{coeffs}.
To see the relation to the various structure theorems in Section~\ref{structure}, we need first to know 
the structure of~$\tM_*(\G)$.  For the moment we do this only for~$\G=\G_1$; the results for general lattices,
which differ in several respects, will be discussed in Section~\ref{cocompact}.  In this case there are
two quite different descriptions of quasimodular forms in terms of modular forms, as already mentioned at 
the beginning of the section.  One of them, which is an obvious consequence of the expressions for the
rings $M_*=M_*(\G_1)$ and $\tM_*=\tM_*(\G_1)$ as $\C[E_4,E_6]$ and $\C[E_2,E_4,E_6]$, respectively, is
that every quasimodular form on the full modular group can be written uniquely as a polynomial in~$E_2$
with modular coefficients.  This gives a canonical isomorphism 
\be\label{isom1} \bigoplus_{r=0}^nM_{K-2r} \;\;\mybig\simeq\;\;\tM_K^{(\le n)}\,, \qquad 
 (g_0,\dots,g_n) \;\;\longleftrightarrow\;\; F\,=\,\sum_{j=0}^ng_j\,E_2^j\,. \ee
The second description, which is equally easy to establish (see Proposition~\ref{QMF.Add} for a
proof in a slightly more general situation), is that every quasimodular form of
positive weight can be expressed uniquely as a linear combination of derivatives of modular forms 
and of~$E_2$.  This means that if we define spaces $M_k^+=M_k^+(\G_1)$ by $M_0^+=\{0\}$, $M_2^+=\C\,E_2$, 
and $M_k^+=M_k$ for $k>2$, then for every $K>0$ we have a second canonical isomorphism
\be\label{isom2} \bigoplus_{r=0}^nM^+_{K-2r} \;\;\mybig\simeq\;\;\tM_K^{(\le n)}\,, \qquad 
(h_0,\dots,h_n) \;\;\longleftrightarrow\;\; F\,=\, \sum_{j=0}^n D^j(h_j) \ee
between $\tM_K^{(\le n)}$ and a space of of the same dimension as before, but in which the 1-dimensional 
space $\C\cdot1$ has been replaced by the 1-dimensional space $\C\cdot E_2$.  (We will see a natural 
interpretation of this in Section~\ref{ERCB2} when we interpret the quasimodular form $E_2/12$ as a modified 
derivative of the modular form~1.)  

If we now combine the two isomorphisms~\eqref{isom1} and~\eqref{isom2} with the isomorphism given in
Theorem~\ref{QMFtoMLDO} between quasimodular forms and MLDOs, we get two of the four descriptions of MLDOs
given in Theorems~\ref{Struct} and~\ref{bcd-Struct} of Section~\ref{structure}: if $F$ has the form given 
in~\eqref{isom1}, then the associated MLDO $L_{F,k}$ is given by the third isomorphism in~\eqref{bcd-exp} 
with $d_r=12^rg_r/(k)_r$, while  if $F$ has the form given in~\eqref{isom2}, then $L_{F,k}(f)$ 
is a linear combination of the Rankin-Cohen brackets $[h_j,f]_j^{(K-2j,k)}$ for $j<(K-2)/2$ and 
$\K_k^j(f)$ for $j=(K-2)/2$, with easily determined coefficients. (See Section~\ref{ERCB2}.)

\section{Cohen-Kuznetsov series, Rankin-Cohen brackets, and higher Serre derivatives}  \label{ERCB}

The best way to understand the relation between the derivatives and the modular transformation properties
of functions in the upper half-plane is through the generating function of their derivatives.  The right 
generating series to use, as was discovered independently by Henri Cohen~\cite{HC} and Nikolai Kuznetsov~\cite{Kuz}
in~1975, is defined by the formula 
\be\label{DefCK}
\Phi_{f,k}(\t,X)  \= \sum_{n=0}^\infty \frac{D^n(f)}{n!\,(k)_n}\,X^n
\ee
for any holomorphic function $f$ in $\H$ and $k\in\Z$ and satisfies the transformation property
\be\label{ModCK}
\Phi_{f,k}\Bigl(\frac{a\t+b}{c\t+d},\frac X{(c\t+d)^2}\Bigr) \= 
(c\t+d)^k\,\exp\Bigl(\frac c{c\t+d}\,\frac X{2\pi i}\Bigr)\,\Phi_{f|_k\g,k}(\t,X)
\ee
for any $\g=\sm abcd\in\SL$.  This is an easy consequence of the transformation formula~\eqref{transf1}
(Lemma~\ref{Transf1}) and is in turn equivalent to the inverse transformation formula~\eqref{transf2}.  
If $f\in M_k(\G)$ for some lattice $\G\subset\SL$, then~\eqref{ModCK} says that the function $\Phi_{f,k}(\t,X)$
is invariant up to a simple automorphy factor under the action 
$(\t,X)\mapsto\bigl(\frac{a\t+b}{c\t+d},\frac X{(c\t+d)^2}\bigr)$ of~$\G$.  In this case we usually
omit the weight, writing simply $\Phi_f(\t,X)$ for~$\Phi_{f,k}(\t,X)$.

The Cohen-Kuznetsov series are related to the Rankin-Cohen brackets defined in~\eqref{rc-bracket} by
\be\label{CKandRCB} \sum_{n=0}^\infty \frac{[f,g]_{\vphantom,_n}^{(k,\ell)}(\t)}{(k)_n\,(\ell)_n}\,X^n 
  \= \Phi_{f,k}(\t,-X)\,\Phi_{g,\ell}(\t,X)\,, \ee
so that the transformation property~\eqref{ModCK} immediately implies\footnote{at least if $k$ 
and~$\ell$ are positive.  If we allow the modular forms $f$ and $g$ to be meromorphic, then their 
weights can be negative; in this case Lemma~\ref{Transf1} still implies that $[f,g]_n^{(k,\ell)}$ 
is modular of weight~$k+\ell+2n$, but this cannot be proved using~\eqref{CKandRCB} since the 
definition of $\Phi_{f,k}$ does not make sense when $k$ is a non-positive integer.} the transformation property
$$ \bigl[f|_k\g,\,g|_\ell\g\bigr]^{(k,\,\ell)}_n 
  \= \bigl[f, g\bigr]^{(k,\,\ell)}_n\,\bigl|_{k+\ell+2n}\g \qquad(\g\in\SL) $$
of Rankin-Cohen brackets and hence that $[M_k(\G),M_\ell(\G)]_n^{(k,\ell)}\subseteq M_{k+\ell+2n}(\G)$.
But they are also the key to understanding the higher Serre derivatives~$\d_k^{[n]}(f)$ and $\K_k^n(f)$
introduced in Section~\ref{RCBandHSD}.  

For this purpose, we introduce the three modified Cohen-Kuznetsov series 
$$ \Phi^{(\Dh)}_{f,k}(\t,X) \= \sum_{n=0}^\infty \frac{\Dh^n(f)}{n!\,(k)_n}\,X^n, \quad
\Phi^{(\d)}_{f,k}(\t,X) \= \sum_{n=0}^\infty \frac{\d_k^{[n]}(f)}{n!\,(k)_n}\,X^n, \quad
\Phi^{(\K)}_{f,k}(\t,X) \= \sum_{n=0}^\infty \frac{\K_k^n(f)}{n!\,(k)_n}\,X^n $$
obtained by replacing $D^n(f)$ in~\eqref{DefCK} by $\Dh_k^{[n]}(f)$, $\d_k^n(f)$ or $\K_k^n(f)$, respectively,
where again the index~$k$ will usually be omitted if $f$ is modular of weight~$k$. 
Here $\Dh_k$ is the non-holomorphic derivative $\Dh_kf(\t)=Df(\t)-\frac k{4\pi y}f(\t)$
($y:=\Im(\t)$) and $\Dh_k^n$ denotes the composition $\Dh_{k+2n-2}\cdots\Dh_{k+2}\Dh_k$.  It is well known, and
easily checked, that $\Dh_k(f|_k\g)=\Dh_k(f)|_{k+2}(f)$, so that $\Dh_k^n$ preserves modularity (but not holomorphy).  
The first two of these modified series were studied in~\cite{123}, where it was shown 
(pp.~54--55) that they are related to the original Cohen-Kuznetsov series $\Phi_{f,k}=\Phi_{f,k}^{(D)}$ by
\be\label{CKrelations}
\Phi^{(D)}_f(\t,X) \= e^{X/4\pi y}\,\Phi^{(\Dh)}_f(\t,X) \= e^{XE_2(\t)/12}\,\Phi^{(\d)}_f(\t,X)\,, 
\ee
where $E_2$ is the non-modular Eisenstein series of weight~2 defined in Section~\ref{basics}.
(We should mention here that our notations differ in several respects from those of~\cite{123}: the derivations 
denoted here by $W$, $\sd=\sd_k$ and $\Dh$ are denoted there by $E$, $\vth=\vth_k$ and~$\p=\p_k$,
or inadvertently on p.~60 by $\vth=\vth_k$, 
the Cohen-Kuznetsov series denoted here by $\Phi_f$, $\Phi_f^{(\Dh)}$ and $\Phi_f^{(\d)}$ 
are denoted in~\cite{123} by $\widetilde f_D$, $\widetilde f_\p$ and $\widetilde f_\d$, 
and the argument~$\t\in\H$ is denoted there by~$z$.)  The second two of these relations give
\be\label{CKrelations2}
 \Phi^{(\d)}(\t,X) \= e^{-X\Eh(\t)/12}\,\Phi^{(\Dh)}(\t,X) \,, 
\ee
where $\Eh(\t)$ denotes the non-holomorphic ``completion" $E_2(\t)-\frac3{\pi y}$ of~$E_2(\t)$,
which transforms like a modular form of weight~2 because of the transformation equation~\eqref{E2transf},
so that the fact that $\d_k^{[n]}$ maps modular forms of weight~$k$ on~$\G_1$ to modular forms 
of weight~$k+2n$ is a direct consequence of the fact that $\Dh_k^n$ preserves modularity.  The 
equality between the first and third terms of~\eqref{CKrelations}, on the other hand, immediately
gives equation~\eqref{defVZexpl}, the first part of Theorem~\ref{VZtoKK}.  

For the proof of the second part of Theorem~\ref{VZtoKK} concerning the Kaneko-Koike derivatives, we 
introduce a Cohen-Kuznetsov series $\Phi_1=\Phi_{1,0}$ for the constant function $1\in M_0(\G_1)$.  
The original definition does not make any sense in this case, because all but the first terms of~\eqref{DefCK} 
have vanishing numerator and denominator.  The correct definition is given by the following proposition,
which was also found by Dai (\!\!\cite{Dai}, eq.~(6.7)).
\begin{proposition} \label{CK1}
The Cohen-Kuznetsov series $\Phi_1=\Phi_{1,0}$ defined by
\be\label{defCK1}  
 \Phi_{1,0}(\t,X) \= 1 \+ \frac X{12}\,\Phi_{E_2,2}(\t,X) 
 \= 1 \+ \frac1{12}\,\sum_{n=1}^\infty \frac{D^{n-1}E_2(\t)}{n!\,(n-1)!}\,X^n \ee
transforms under the action of $\sm abcd\in\G_1$ by
\be\label{CK1transf} \Phi_1\Bigl(\frac{a\t+b}{c\t+d},\,\frac X{(c\t+d)^2}\Bigr) \= 
    \exp\Bigl(\frac c{c\t+d}\,\frac X{2\pi i}\Bigr)\,\Phi_1(\t,X)\,.  \ee
The modified Cohen-Kuznetsov series $\Phi^{[\d]}_1=\Phi^{[\d]}_{1,0}$ defined by
\be\label{defCKVZ1}  
 \Phi^{[\d]}_1(\t,X) \= e^{-XE_2(\t)/12}\,\Phi_1(\t,X)
\ee
is invariant under $(\t,X)\mapsto\bigl(\frac{a\t+b}{c\t+d},\frac X{(c\t+d)^2}\bigr)$.
\end{proposition} 
\begin{proof}
We can prove this by imitating any of the three proofs given in~\cite{123} (p.~54) for equation~\eqref{ModCK}
when the weight~$k$ is positive.  One of them is the one we gave above using equation~\eqref{transf2}, and 
we can apply this also in weight~0 by using equation~\eqref{transf2} for $f=E_2$ and $k=2$ together with the 
transformation formula~\eqref{E2transf} to deduce that $\Phi_{E_2,2}$ has precisely the transformation properties 
under~$\G_1$ needed to imply equation~\eqref{CK1transf} for $\Phi_{1,0}$.
Another proof in~\cite{123} used the fact that $\Phi_f(\t,X)$ for $k>0$ is the unique solution of the 
differential equation $\bigl(X\frac{\p^2}{\p X^2}+k\frac\p{\p X}-D)\Phi_f=0$ with constant term~$f$ and 
that the function $(c\t+d)^{-k}e^{-cX/2\pi i(c\t+d)}\Phi_f\bigl(\frac{a\t+b}{c\t+d},\frac X{(c\t+d)^2}\bigr)$ 
has the same properties. For~$f=1$ and $k=0$ the first statement is no longer true, but all solutions of 
the differential equation $\bigl(X\frac{\p^2}{\p X^2}-D)\Phi_f=0$ with constant term~1
have the form $1+\sum_{n\ge1}\frac{D^{n-1}(F)}{n!(n-1)!}X^{n-1}$ for some function $F$, which is
arbitrary, and therefore for any~$F$ there is a unique solution beginning $1+FX+\text O(X^2)$.
This function therefore satisfies the transformation equation~\eqref{CK1transf} if (and only if)
$F$ satisfies $F|_2\g=F+\frac1{2\pi i}\frac c{c\t+d}$, which is the case for $F=\frac1{12}E_2$
by virtue of~\eqref{E2transf} (and for no other holomorphic function~$F$ of bounded growth at
infinity because there are no non-zero modular forms of weight~2 on~$\G_1$). The second statement
of the proposition follows immediately from the first and the transformation property~\eqref{E2transf} 
of~$E_2(\t)$. We observe that the definition~\eqref{defCKVZ1} is just the relation given
in~\eqref{CKrelations} between $\Phi^{[\d]}_f$ and $\Phi_f$ for modular forms~$f$ of positive weight,
but now decreed to hold also for~$f=1$.
\end{proof}

By combining the definition~\eqref{KKdef} of the Kaneko-Koike operator with equations~\eqref{CKandRCB}
and~\eqref{defCK1}, we can rewrite that definition in terms of generating series as
\be\label{CKrelations3} 
 \Phi_{f,k}^{(\K)}(\t,X) \= \Phi_{f,k}^{(D)}(\t,X)\,\Phi_{1,0}(\t,-X)\,.\ee
Comparing this with equation~\eqref{CKandRCB}, we see that the Kaneko-Koike derivatives $\K^n(f)$ can be
interpreted in some sense simply as the Rankin-Cohen brackets of $f$ with the constant function~1.  (We will 
make more precise sense of this statement in the next section.)  In any case, equation~\eqref{CKrelations3} 
together with the transformation property~\eqref{CK1transf} of $\Phi_1$ immediately implies the fact
that $\K_k^n$ is an MLDO of type $(k,k+2n)$ and hence sends modular forms of weight~$k$ to modular 
forms of weight~$k+2n$.  

Equation~\eqref{CKrelations3} tells us that the fourth of our modified Cohen-Kuznetsov series $\Phi_{f,k}^{(\K)}$
is related to the other three by universal factors, not depending on the form~$f$ or its weight~$k$, but 
no longer purely exponential in~$X$ as was the case for the relationships among the other three series.
We can combine it with the previous statements~\eqref{CKrelations} and~\eqref{CKrelations2} into a single diagram
\be\label{Diagram}
\begin{tikzcd}[cells={nodes={draw},outer sep=5pt}]
  & \Phi_f^{(\Dh)}(\t,X)
  \arrow[dr,"\displaystyle\cdot\;e^{-X\Eh(\t)/12}"] & \\[2mm]
  \Phi_f^{(D)}(\t,X)
  \arrow[ur, "\displaystyle\cdot\;e^{X/4\pi y}" anchor=south east]
  \arrow[rr,"\displaystyle\cdot\;e^{-XE_2(\t)/12}"]
 \arrow[dr,"\displaystyle\cdot\;{\Phi_1^{(D)}(\t,-X)}" anchor=north east] & 
  & \Phi_f^{(\d)}(\t,X)
  \arrow[dl,"\displaystyle\cdot\;{\Phi_1^{(\d)}(\t,-X)}"] \\[2mm]
  & \Phi_f^{(\K)}(\t,X) & 
\end{tikzcd}
\ee
showing the relationships between all four series.  The last arrow of this diagram also proves
equation~\eqref{KKfromVZ} of Theorem~\ref{VZtoKK}, with the modular forms $\om_m$ defined by the generating function
\be\label{omDef} 
 \sum_{m=0}^\infty \frac{\om_m(\t)}{{m!}^{\,2}}\,(-X)^m \= \Phi^{(\d)}_1(\t,X)\,. 
\ee
The first terms of this series are then easily calculated to be the ones tabulated in Theorem~\ref{VZtoKK}.

\section{Modified derivatives and extended Rankin-Cohen brackets}  \label{ERCB2}

In this section we introduce a new and very convenient notion that will make many
statements of the theory simpler and more uniform, namely, modified higher derivatives 
$\DD n(f)$ that are simply a renormalization of the usual derivatives $D^n(f)$ when $f$ has positive
weght but which are non-trivial also for~$f=1$.  

Specifically, if $f$ is a modular form of positive weight we define
\be\label{DDnPosWeight}
 \DD n(f) \= \DD n_k(f) \= \frac{D^n(f)}{(k)_n} \qquad(f\in M_k(\G),\;k>0)\,,
\ee
where $(k)_n=k(k+1)\cdots(k+n-1)$ as before.  This renormalization seems unnecessary and somewhat 
strange at first sight but turns out to be very natural.  In particular, in terms of this notation 
the Cohen-Kuznetsov series~\eqref{DefCK} can be written very simply as
\be\label{CKviaDD}
 \Phi_f(\t,X) \= \sum_{n=0}^\infty \frac{\DD n(f)}{n!}\,X^n
\ee
or symbolically as
\be\label{CKsymbolic}
 \Phi_f(\t,X) \= e\up{XD}f(\t)\,.
\ee
The real point, however, is that we can now define higher derivatives of the 
constant function~1 in a useful way.  The definition~\eqref{DDnPosWeight} makes
no sense in this case for $n>0$ (for $n=0$ it of course just gives $D\up0(1)=1$),
because both the numerator and the denominator of the fraction vanish.  But 
if we recall that the Serre derivative $\d_k(f)=D(f)-\frac k{12}E_2f$ of a modular
form of weight~$k$ is modular of weight~$k+2$, and that the function 
$\frac{\d_k(f)}k=\DD 1(f)-\frac{E_2}{12}f$ is therefore also modular of this weight, we see
that $\DD 1(1)$ should be defined as the sum of $\frac1{12}E_2$ and a modular form 
of weight~2, and hence must be taken to be~$\frac1{12}E_2$.  For the higher derivatives,
there is then no problem, since equation~\eqref{DDnPosWeight} in the positive
weight case gives the inductive formula $\DD {n+1}f=D(\DD n(f))/(k+n)$, which
makes sense even for $k=0$ as soon as~$n>0$. We therefore define $\DD n(1)$ for all $n\ge0$ by
\be\label{DDn1}
 \DD n(1) \= D_0\up n(1) \= \begin{cases} \phantom{XX}1 & \text{if $n=0$,} \\
  \frac{D^{n-1}(E_2)}{12\,(n-1)!} & \text{if $n\ge1$,} \end{cases} 
\ee
and discover that formula~\eqref{CKviaDD} for the Cohen-Kuznetsov series of
a modular form of positive weight is still valid for $f=1$ and~$k=0$ if we use the definition 
of $\Phi_1$ given in Proposition~\ref{CK1}. In fact, we could have simply used equations~\eqref{CKviaDD} 
and~\eqref{defCK1} as the motivation of the definition~\eqref{DDn1}.  We can also give $\DD n(1)$ by 
the explicit formula
\be\label{Dn1}
  \DD n(1) \= \sum_{m=0}^n(-1)^m\binom nm\,\frac{\om_m}{m!}\,\Bigl(\frac{E_2}{12}\Bigr)^{n-m}
\ee
with $\om_m$ as in Theorem~\ref{VZtoKK}.  (The proof will be given later.)

So far this is purely formal and a matter of introducing new notations for objects that we 
already knew.  But in fact it immediately leads to many simplifications in the formulas 
that we have studied so far, as well as in those given in the rest of this section and 
the two following ones.  The first concerns the Rankin-Cohen bracket. If we rewrite the 
definition~\eqref{rc-bracket} in terms of the modified derivatives~\eqref{DDnPosWeight} 
for $f$ and $g$ of positive weights~$k$ and~$\ell$, we find that  
\be\label{proportional}
\bigl[f,\,g\bigr]^{(k,\,\ell)}_n
  \=\frac{(k)_n\,(\ell)_n}{n!}\,\sum_{i=0}^{n}(-1)^i\binom ni\DD i(f) \,\DD {n-i}(g)\,. 
\ee
This equation still holds if $k$ or~$\ell$ is zero, but is then not interesting since 
both sides vanish for~$n\ge1$, because the factor in front of the sum is then equal to~0.
But the sum itself is in general non-zero and therefore give an interesting modified
Rankin-Cohen bracket even in the case of forms of weight~0.  We therefore define
the~$n$\thin th {\it extended Rankin-Cohen bracket} $\langle f, g\rangle_n$ 
of any two modular forms $f$ and~$g$ and any $n\ge0$ by
\be\label{defERCB}  
\big\langle f,\,g\big\rangle_n \: \sum_{i=0}^{n}(-1)^i\binom ni\DD i(f) \,\DD {n-i}(g)\,.
\ee
We then have
\begin{theorem} \label{prop.ERCB}
If $f$ and~$g$ are modular forms of weight $k$ and~$\ell$, respectively, then the
extended Rankin-Cohen bracket~\eqref{defERCB} is a modular form of weight~$k+\ell+2n$
for all~$n\ge0$, and is given in terms of the usual Rankin-Cohen bracket, the
Kaneko-Koike derivative, and the modular forms $\om_m$ defined in Theorem~\ref{VZtoKK} by 
\be\label{ERCBexplicit}  
 \big\langle f,\,g\big\rangle_n \= \begin{cases} 
    \qquad \quad\, \frac{n!}{(k)_n(\ell)_n}\,[f,g]^{(k,\ell)}_n &\text{if $k,\,\ell>0$,}\\
   \qquad\quad\qquad \frac1{(\ell)_n}\,\K_\ell^n(g)  &\text{if $f=1$, $\ell>0$,}  \\
  \frac1{n!}\,\sum_{m=0}^n(-1)^m\,{\binom nm}^2\,\om_m\,\om_{n-m}  &\text{if $f=g=1$.} \end{cases}
\ee
\end{theorem} 
\begin{proof} The statement about modularity follows immediately from the transformation
properties of $\Phi_f(\t,X)$ for modular forms of positive weight (equation~\eqref{ModCK})
or weight~0 (equation~\eqref{CK1transf}) together with equation~\eqref{CKandRCB}, which
in the new notation and in view of equation~\eqref{CKviaDD} takes on the simpler form
\be\label{CKandERCB}
 \Phi_f(\t,-X)\,\Phi_g(\t,X)\,\=\sum_{n=0}^\infty \langle f,g\rangle_n\,\frac{X^n}{n!}\,. \ee
The proportionality of $\langle f, g\rangle_n$ with $[f,g]_n$ for $k$~and~$\ell$ positive follows 
from equation~\eqref{proportional}, which says that $[f,g]_n=\frac{(k)_n(\ell)_n}{n!}\langle f,g\rangle_n$
in this case.  For $g$ of positive weight~$\ell$ and $f=1$, equations~\eqref{CKandERCB} and~\eqref{CKrelations3}
immediately imply that $\langle f,g\rangle_n=\frac1{(\ell)_n}\K_\ell^n(g)$, so that the Kaneko-Koike derivatives 
can be seen as a special case of the Rankin-Cohen bracket once the latter has been extended.  Finally, the explicit
formula for $\langle 1,1\rangle_n$ follows immediately from~\eqref{CKandERCB} and~\eqref{omDef}.
\end{proof}

For completeness, we give a small table of values of~$\langle 1,1\rangle_n$ for $n$~even (the odd values 
vanish because in that case the bracket is antisymmetric), with $\D=(E_4^3-E_6^2)/1728$ as usual:
\begin{center}  \def\arraystretch{1.5}  \begin{tabular}{c|ccccccccc}
$n$ & $\;0\;$ & $\;2\,$ & $\;4\;$ & $\;6\;$ & $\;8\;$ & $\;10\;$ & $\;12\;$  \\ \hline
$n!\,\langle1,1\rangle_n$ & $\;1\;$ & $-\frac1{36}E_4$ & $0$ &  $\;36\D\;$  & $\;\frac{352}3E_4\D\;$ & $\;260E_4^2\D\;$  
 & $480E_4^3\D +1259136\D^2$   \end{tabular} \;. \end{center} 

We emphasize that the point of the extended bracket is not to extend the definition of 
the original Rankin-Cohen bracket from positive weight to weight~0, since~\eqref{rc-bracket} already
makes perfectly good sense for modular forms of arbitrary positive, zero, or negative integral weights (and for that 
matter even half-integral or arbitrary rational weights) and always sends a pair of modular forms of weights 
$k$ and~$\ell$ to a modular form of weight~$k+\ell+2n$, simply because the transformation laws~\eqref{transf1} 
and~\eqref{transf2} hold for any value of~$k$.  But if one is interested in studying holomorphic 
modular forms only, then there are no non-zero forms of negative weights and the only forms of weight~0 are 
constants, and in the  case when $f$ or $g$ is constant the original Rankin-Cohen bracket 
$[f,g]_n$ vanishes, whereas the new bracket $\langle f,g\rangle_n$ does not. However, the real point
of introducing~$\lr{f,g}_n$ is not even this extension to weight~0 but rather that all formulas involving 
Rankin-Cohen brackets become much simpler even in positive weight if we use the extended ones instead, as
will see repeatedly in the rest of this paper.

An immediate consequence of the above definition and theorem is that Theorem~\ref{Struct}, which in its
original formulation involved an unaesthetic case distinction depending whether the order of the MLDO in question 
was less than half its weight (in which case it was a combination of Rankin-Cohen brackets) or equal to half 
its weight (in which case one needed the Kaneko-Koike operators as well), can now be given in a more 
uniform and much simpler form:
\begin{theorem}[= uniform restatement of Theorem~\ref{Struct}]\label{uniformMLDO} Every MLDO on the 
full modular group can be written uniquely as a sum of extended Rankin-Cohen brackets. 
\end{theorem}
\noindent Explicitly, if $L\in\MLDO_{k,k+K}^{(\le n)}$, then $L(f)=\sum_{r=0}^n\langle h_r,f\rangle_r$ 
for some modular forms $h_r\in M_{K-2r}$. 

A different use that we can make of the modified derivatives is to give a cleaner and more uniform description
of quasimodular forms for the full modular group (and in fact also for other non-cocompact lattices,
as we will see in the next section), which then also has applications to the description of MLDOs.
In Section~\ref{MLDOandQMF} we saw that the space $M_K^{(\le n)}(\G_1)$ of quasimodular forms 
of strictly positive weight~$K$ and depth $\le n$ is canonically isomorphic to $\bigoplus_{0\le r\le n}M_{K-2r}(\G_1)$
in two different ways: on the one hand by writing its elements as polynomials in~$E_2$ (equation~\eqref{isom1}) and 
on the other by writing them as linear combinations of $r$\thin th derivatives of elements of spaces $M_k^+(\G_1)$ 
(where $K=k+2r$) that are somewhat artificially defined as $M_k(\G_1)$ for $k>2$ and as $\qM_2(\G_1)=\C E_2$ 
for~$k=2$ (equation~\eqref{isom2}). But using the renormalized derivatives $\DD n$ and their extension~\eqref{DDn1}
to weight~0, we can write the latter isomorphism much more naturally by saying that each quasimodular form of
depth~$\le n$ is a linear combination of modified derivatives of modular forms of weight~$\ge0$, namely,
\be\label{isom3} \bigoplus_{r=0}^nM_{K-2r} \;\;\mybig\simeq\;\;\tM_K^{(\le n)}\,, \qquad 
(h_0,\dots,h_n) \;\;\longleftrightarrow\;\; F\,=\, \sum_{r=0}^n \DD r(h_r) \ee
instead of~\eqref{isom2}. Of course the content of this isomorphism is the same as that of~\eqref{isom2}, 
but the new one is better in several ways.  One is that if we now two isomorphisms between the ring
$\qM_*$ of quasimodular forms on~$\G_1$ and the {\it same} ring $M_*[X]$ of polynomials 
in one variable over the ring $M_*$ of modular forms on~$\G_1$, one by sending the polynomial
$\sum h_rX^r$ to the quasimodular form $\sum\DD r(h_r)$ and one by sending the polynomial $\sum g_rX^r$
to the quasimodular form $\sum g_r(E_2/12)^r$.  (Note that it is more natural here and in all other formulas 
to use $E_2/12$ rather than $E_2$ itself as the generator of $\qM_*(\G_1)$ over $M_*(\G_1)$, because $E_2/12$
is mapped to~1 under the derivation~$\de$ on~$\qM_*$.  This comment will become more important in the next 
section, when we replace $E_2/12$ by a choice of element $\E\in\qM_2(\G)$ with $\de(\E)=1$ for sublattices $
\G$ of $\G_1$ or more general lattices~$\G$ with cusps.)  If we compose one of these isomorphisms with
the inverse of the other, we get an endomorphism of $M_*[X]=\bigoplus M_*X^r$ that is represented 
with respect to the basis $\{X^r\}$ by a triangular matrix with 1's on the diagonal (i.e., the coefficient
of $(E_2/12)^n$ in the expansion of $\DD n(f)$ as a polynomial in $E_2/12$ is equal to~$f$, with coefficient~1).
In fact, we can now describe this whole matrix completely explicitly by using the generating
series identities~\eqref{CKrelations} and~\eqref{defCKVZ1}, which imply the formula
\be \label{DDntoEVZ}
\DD n(f) \= \sum_{r=0}^n \binom nr\;\d^{\langle r\rangle}(f)\;\Bigl(\frac{E_2}{12}\Bigr)^{n-r} 
\ee
(equivalent to~\eqref{defVZexpl} when the weight of~$f$ is positive, but now considerably simpler) for any $f\in M_*$.  
Here we have defined $\d^{\langle r\rangle}(f)$ by modifying the canonical higher Serre derivatives defined in 
Section~\ref{RCBandHSD} in the same way as we did for the ordinary derivatives, i.e., by setting
\be\label{VZmodified}
 \d^{\langle n\rangle}(f) \= \d^{\langle n\rangle}_k(f) \= \frac{\d^{[n]}(f)}{(k)_n} 
   \quad(f\in M_k(\G),\;k>0),\qquad \d^{\langle n\rangle}(1) \= \frac{(-1)^n}{n!}\,\om_n\,,
\ee
so that $\Phi_f^{[\d]}(\t,X)=\sum \d^{\langle n\rangle}(f)\,\frac{X^n}{n!}$ in all cases.
Note that, just as the unnormalized canonical higher Serre derivatives can be used instead of ordinary
derivatives to write the Rankin-Cohen brackets as a sum of terms that are individually modular 
(equation~\eqref{rc-bracket-vz}), the modified ones can be used to write the extended Rankin-Cohen bracket
$\langle f,g\rangle_n$ as a sum of $n+1$ terms each of which is modular:
\be\label{RCBviaEVZ}
\big\langle f,\,g\big\rangle_n \= \sum_{r=0}^{n}(-1)^r\binom nr\,\d^{\langle r\rangle}(f)\,\d^{\langle n-r\rangle}(g)\,.
\ee
Here, just as with~\eqref{rc-bracket} and~\eqref{rc-bracket-vz}, the coefficients are the same as in~\eqref{defERCB}.

Finally, the discussion above can also be used to give a new interpretation of the operator~\eqref{pairing} 
defined in Theorem~\ref{QMFtoMLDO} in the case of its action on modular forms, and at the same time to extend it to 
the case when~$k=0$. Specifically, we define a {\it pairing}
\be\label{newpairing}  \bigl\{\;\,,\;\,\bigr\}:\,M_*\otimes \qM_* \ra M_*\,, \qquad
   f\otimes F\;\mapsto\;\bigl\{f,F\bigr\} \= \sum_{r=0}^\infty \frac{(-1)^r}{r!}\,\de^r(F)\,\DD r(f)\,. 
\ee
between quasimodular and modular forms.  The bracket $\{f,F\}$ coincides with $L_{F,k}(f)$ as defined in
Theorem~\ref{QMFtoMLDO} when $k$ is strictly positive, and is therefore modular (of weight~$k+K$) because 
$L_{F,k}$ is an MLDO of type~$(k,k+K)$, but it now makes sense also in weight~0 and is still modular in that 
case.  To see this, and to understand the pairing better, we observe that
\be\label{compatible}
\bigl\{f\thin,\thin \DD n(g)\bigr\} \= \bigl\langle f,\thin g\bigr\rangle_n
\ee
for any modular forms~$f$ and~$g$ and any integer~$n\ge0$ (this follows immediately from equation~\eqref{drDn} 
in the next section), so that the pairing $f\otimes F\mapsto \{f,F\}$ is related to the decomposition of quasimodular 
forms given in equation~\eqref{isom3} by
\be\label{pairingandRCB}
 F \= \sum_{r=0}^n\DD r(g_r) \quad\Longrightarrow\quad \bigl\{f,\thin F\}
      \=  \sum_{r=0}^n \bigl\langle f,\thin g_r\bigr\rangle_r\,. 
\ee
This also gives a proof of the formula~\eqref{Dn1}, because under the correspondence~\eqref{isomQMtoMLDO}
between quasimodular forms and modular linear differential operators the MLDOs $\K_k^n$ and $\d_k^{[n]}$ correspond to the quasimodular forms~$(-1)^n(k)_n\DD n(1)$ and $(k)_n\thin\bigl(-\frac1{12}E_2\bigr)^n$, respectively, so 
that~\eqref{Dn1} is just a consequence of~\eqref{KKfromVZ}, which was proved in the last section. 

As a final remark, we observe that in equations~\eqref{CKviaDD}, \eqref{defERCB},~\eqref{RCBviaEVZ} 
and~\eqref{newpairing} we were able to omit the index~$k$ that was needed in the corresponding earlier equations~\eqref{DefCK},
\eqref{rc-bracket},~\eqref{rc-bracket-vz} and~\eqref{pairing} because it is already incorporated into the definition 
of the modified derivative~$\DD r$.  A consequence of this is that these formulas can be written symbolically in a very
simple form.  This was already done in the first case in equation~\eqref{CKsymbolic}, and the other three can be written as
$$ \big\langle f,\,g\big\rangle_n \= {\sf m}\bigl((1\otimes D\m D\otimes1)\up n\,(f\otimes g)\bigr)
  \= {\sf m}\bigl((1\otimes\d\m\d\otimes1)\up n\,(f\otimes g)\bigr)$$
and
$$\bigl\{f,F\bigr\} \= {\sf m}\bigl(e\up{-D\otimes\de}(f\otimes F)\bigr)\,, $$
where $\sf m$~denotes multiplication.  Alternatively, we can also express the right-hand sides of these last three
equations as the restrictions to the diagonal $\t_1=\t_2=\t$ of 
$(D_2-D_1)\up n(f(\t_1)g(\t_2))$, $(\d_2-\d_1)\up n(f(\t_1)g(\t_2))$ or $e\up{-\de_1D_2}(f(\t_1)g(\t_2))$, respectively, 
where the subscripts on the differential operators indicate which variable~$\t_i$ they act on.

\section{Application: higher Serre derivatives of quasimodular forms}\label{APPLIC}

The theme of this section is a non-trivial extension of the Rankin-Cohen 
bracket from modular to quasimodular forms that was discovered by Fran\c cois Martin and Emmanuel Royer 
and in a different form by Youngju Choie and Min Ho Lee (and that had also been found by the third 
author, but never published).  Of course the Rankin-Cohen bracket $[f,g]_n^{(k,\ell)}$ can be defined 
for any two holomorphic functions $f$ and~$g$, any non-negative integer~$n$, and any integers (or for
that matter, even complex numbers) $k$~and~$\ell$, but the original point of the specific complicated-looking
bilinear combination of derivatives in its definition was that if $f$ and~$g$ are modular forms of weights 
$k$ and~$\ell$ on some lattice, then $[f,g]_n^{(k,\ell)}$ is also a modular form, of weight~$k+\ell+2n$.
At first sight this statement seems impossible to generalize in an interesting way to quasimodular forms,
since the ring of quasimodular forms is closed under differentiation anyway, so that each term of~\eqref{rc-bracket}
is quasimodular if $f$ and~$g$ are.  But if we remember that modular forms are simply quasimodular forms of
depth~0, then we can ask if there is a way to make a bracket of two quasimodular forms whose depth is at
most the sum of the two individual depths, independent of~$n$.  (The individual terms of~\eqref{rc-bracket}
in general have depths equal to the sum of the two individual depths plus~$n$.)  The result of Martin and 
Royer is that this {\it is} possible, and that all one has to do is to replace the upper indices $k$ and $\ell$
of the bracket by $k-p$ and $\ell-q$, where $p$ and~$q$ are the depths (or upper bounds on the depths)
of $f$ and~$g$.  We will prove this in a somewhat more general form by replacing the original Rankin-Cohen 
brackets with the extended ones that were defined in the last section, in  which case the theorem still makes
sense (and still is true) even if one of $f$ or $g$ has weight~0 and hence is constant.  (If both have weight~0,
then there is nothing to prove, since the depths are then also~0 and we are simply back to the modularity of
the extended Rankin-Cohen brackets $\langle 1,1\rangle_n^{(0,0)}$.)

Actually, as well as this extension of the Martin-Royer theorem to include forms of weight~0, we 
will give a strictly stronger result of which that theorem is an immediate corollary.  This is the
statement that the action on quasimodular forms of canonical higher Serre derivatives of arbitrary 
orders preserves their depth if one chooses the index of the Serre derivative to be the weight of 
the form  minus its depth, rather than just its weight as in the modular case.  Because of the 
expression~\eqref{rc-bracket-vz} for Rankin-Cohen brackets as bilinear combinations of  canonical 
higher Serre derivatives, this immediately gives the Martin-Royer theorem as well, but it is both 
stronger and simpler, since it applies to individual quasimodular forms rather than to pairs.  

We formulate the two results just described more quantitatively in the following two theorems. 

\begin{theorem}\label{VZforQMF} 
If $f$ is a quasimodular form of weight~$k$ and depth $\le p\,$, then the quasimodular form 
$\d_{k-p}^{[n]}(f)$ has depth $\,\le p\,$ for all integers~$\,n\ge0\,$.  
\end{theorem}

\begin{corollary*} [Generalized Martin-Royer theorem] \label{MRforQMF}
If $f$ and $g$ belong to $\qM_k^{(\le p)}$ and $\qM_\ell^{(\le q)}$, respectively,
then the extended Rankin-Cohen bracket $\langle f,g\rangle_n^{(k-p,\ell-q)}$ 
belongs to $\qM_{k+\ell+2n}^{(\le p+q)}$ for all~$n\ge0$.
\end{corollary*}

For the proofs we will need the following lemma, which is a general statement about the action
of the Lie algebra~$\lie$ that has many applications in the theory of quasimodular forms.

\bl\label{GSM.L1}
 In the universal enveloping algebra of~$\lie$ we have the identities
\be\label{GSM25}\de D^n\= D^{n}\de \+ n\,D^{n-1}(W+n-1)
\ee 
for all $n\geq0$ and more generally 
\be\label{GSM26}
   \de^{\,r}D^n \= \sum_{j=0}^r\binom rj\,(n-j+1)_j\,D^{n-j}\de^{\,r-j}(W+n-r)_j
\ee 
for all $r$, $n\geq 0$, where $(n-j+1)_j\,D^{n-j}$ is to be interpreted as~$0$ if $j>n$. 
\el

\begin{corollary*}
If $f$ is a modular form of weight~$k$, then we have $\de D^n(f)= n(k+n-1)D^{n-1}(f)$ for all $n\ge0$ 
and more generally $\de^{\,r}D^n(f)=(n-r+1)_r(k+n-r)_r\,D^{n-r}(f)$ for all $n,r\ge0$.
\end{corollary*}
\noindent The proof of the lemma (using induction on~$n$ for the first statement and then  
on~$s$ for the second) is straightforward and well known, so we will omit it, and the corollary follows 
immediately since modular forms are annihilated by~$\de$. In terms of the modified derivatives $\DD n(f)$ 
introduced in the last section, we can rewrite the second statement of the corollary as
\be\label{drDn}
\frac{\de^{\,r}(\DD n(f))}{r!}  \= \binom nr\,\DD {n-r}(f) \qquad(n,r\ge0,\;f\in M_*(\G))
\ee
if the weight of~$f$ is positive. Note that this equation remains true also for~$f=1$, 
as one sees by the following calculation using the definition~\eqref{DDn1} and equation~\eqref{GSM26}
applied to~$\E=E_2/12\,$:
\bas & (n-1)!\,\de^r(\DD n(1)) \= \de^r(D^{n-1}(\E)) 
  \=\sum_{j=0}^r \binom rj\,(n-j)_j\,(n-r+1)_j\,D^{n-1-j}(\de^{\,r-j}(\E))
 \\  &\qquad\= (n-r)_r\,(n-r+1)_r\,D^{n-r-1}(\E) \+ r\,(n-r+1)_{r-1}\,(n-r+1)_{r-1}\,D^{n-r}(1) \\
 &\qquad\= \frac{(n-1)!\,n!}{(n-r-1)!\,(n-r)!}\,D^{n-r-1}(\E) \+ r\,\frac{(n-1)!^2}{(n-r)!^2}\,\delta_{n,r}
\= \frac{(n-1)!\,n!}{(n-r)!}\,\DD {n-r}(1)\,.
 \eas
(Actually, to prove~\eqref{drDn} it is enough to give the calculation for~$r=1$, which is slightly 
simpler, and then use induction on~$r$, but we gave the calculation in general because it is not much 
longer and gives a nice illustration of the properties of the modified derivative~$\DD n$.) We already used 
equation~\eqref{drDn} in Section~\ref{ERCB2} to get the formula~\eqref{compatible} relating
extended Rankin-Cohen brackets to the pairing~\eqref{newpairing}.

\smallskip

We now proceed to the proof of Theorem~\ref{VZforQMF}, after which, as already mentioned,
the corollary follows immediately as a consequence of equation~\eqref{rc-bracket-vz} 
(or of its extension~\eqref{RCBviaEVZ} to extended Rankin-Cohen brackets in the case when one
of the forms has weight~0). To do this,
we consider the Cohen-Kuznetsov series of $f$ with index $k-p$, where $f\in\qM_k^{(\le p)}$ as in the theorem.
More specifically, we consider both the $\d$- and the $D$-versions of this series, defined and related by
$$ \Phi^{(\d)}_{f,k-p}(X)\= \sum_{n=0}^\infty \frac{\d_{k-p}^{[n]}(f)}{(k-p)_n\,n!}\,X^n
\= e^{-X\E}\,\Phi_{f,k-p}(X) \= e^{-X\E}\,\sum_{n=0}^\infty \frac{D^n(f)}{(k-p)_n\,n!}\,X^n$$ 
with $\E=E_2/12$ and $\d_{k-p}^{\langle n\rangle}(f)$ defined as in equation~\eqref{VZmodified}.  (Here and for the 
rest of the proof we omit the argument~$\t$ for notational simplicity.) Notice, by the way, that we can assume
that $k>0$, since if $k=0$ then $p$ is also~0 and there is nothing to prove, and then $p\le k/2<k$, so that
the factors $(k-p)_n$ in the denominators of the above formulas never vanish. To prove the theorem, we have 
to show that the series on the left is annhilated by~$\de^{p+1}$, or equivalently that all of the coefficients 
of its image under $\de^p$ are modular rather than merely quasimodular forms.  By Leibniz's formula and
the fact that $\de(\E)=1$ we have
\bas e^{X\E} \de^m\bigl(\Phi^{(\d)}_{f,k-p}(X)\bigr) &\= e^{X\E} \de^m\bigl(e^{-X\E}\,\Phi_{f,k-p}(X)\bigr) \\
&\= \sum_{s=0}^m\binom ms\,(-X)^{m-s}\,\de^s\bigl(\Phi_{f,k-p}(X)\bigr) \= (\de-X)^m\bigl(\Phi_{f,k-p}(X)\bigr) 
\eas
for every $m\ge0$, and we want to show that this expression vanishes for~$m>p$.  

In fact, we will give two different arguments, omitting a few of the details of the calculation in the second case. 
For the first argument we use the fact that every quasimodular form is a linear combination of modified
derivatives $\DD r(h)$ of modular forms (equation~\eqref{isom3}).  Since such a derivative has depth
exactly~$r$, we can assume that the quasimodular form $f$ in Theorem~\ref{VZforQMF} has the form $\DD r(h)$
for some $r\le p$ and some modular form~$h$ of weight $k-2r$.  Then we have
$$ \Phi_{f,k-p}(X) \=\sum_{n=0}^\infty \frac{D^n(\DD r(h))}{(k-p)_n\,n!}\,X^n
\= \frac1{(k-p)_{p-r}} \sum_{n=0}^\infty (k+n-p)_{p-r}\,\DD {n+r}(h)\,\frac{X^n}{n!}  $$
and hence
\bas & (k-p)_{p-r}\,e^{X\E}\,\de^m\bigl(\Phi^{(\d)}_{f,k-p}(X)\bigr) 
 \= \sum_{s=0}^m\binom ms\,(-X)^{m-s}\,\de^s\Biggl(\sum_{n=0}^\infty (k+n-p)_{p-r}\,\DD {n+r}(h)\,\frac{X^n}{n!}\Biggr) \\  
 &\qquad\; \=  X^{m-r}\,\sum_{\ell=0}^\infty
  \Biggl[\sum_{s=0}^m\binom ms\,(-1)^{m-s}\,(k-p+s+\ell-r)_{p-r}(s+\ell-r+1)_r \Biggr]\,\DD \ell(h)\,\frac{X^\ell}{\ell!} 
\eas
for any $m\ge r$, where to obtain the second line we have used equation~\eqref{drDn}.
  If $m>p$ then the expression in square brackets vanishes for every~$\ell$ because it is the $m$\thin th 
difference of a polynomial in~$s$ of degree~$p$.  This proves the theorem.  We also get an explicit 
formula for the $p$\thin th derivative of $\d_{k-p}^{[n]}(f)$ as a modular form for every~$n$, since the $p$\thin th 
derivative of a monic polynomial of degree~$p$ is $p!$ and hence
$$(k-p)_{p-r}\,\de^p\bigl(\Phi^{(\d)}_{f,k-p}(X)\bigr) 
  \=  p!\,X^{p-r}e^{-X\E}\,\sum_{h=0}^\infty \DD h(f)\,\frac{X^h}{h!} \= p!\,X^{p-r}\, \Phi^{(\d)}_f(X)\,, $$
so that $\de^p(\d^{[n]}(f))$ vanishes if $n<p-r$ and is a simple multiple of $\,\d^{[n+r-p]}(h)\,$ if~$n\ge p-r$.

For the second argument, which we only sketch, we work directly with the action of~$\lie$ on the 
various Cohen-Kuznetsov series involved.  This approach involves slightly more calculation 
but has the advantages that it does not use the decomposition~\eqref{isom3} or the special quasimodular
form~\hbox{$\E=E_2/12$}.  For any quasimodular form $f\in\qM_k$ and any positive integer~$K$ we have
\bas (\de-X)\,\Phi_{f,K}(X)  &\=  \sum_{n=0}^\infty \frac{\de(D^n(f))\m n(K+n-1)D^{n-1}(f)}{n!\,(K)_n}\,X^n \\ 
   &\=   \sum_{n=0}^\infty \frac{D^n(\de(f)) \+ n(k-K)\,D^{n-1}(f)}{n!\,(K)_n}\,X^n
      \qquad\quad \text{(by eq.~\eqref{GSM25})} \\  &\= \Phi_{\de(f),K}(X) \+ \frac{k-K}K\,X\,\Phi_{f,K+1}(X)
\eas
and hence by induction on~$m$
$$  (\de-X)^m\Phi_{f,K}(X) 
 \= \sum_{s=0}^m\binom ms\,\frac{(k-K-m+1)_{m-s}}{(K)_{m-s}}\,X^{m-s}\,\Phi_{\de^s(f),K+m-s}(X) $$
for every integer~$m\ge0$.  (We omit the details of this step, which are slightly messy.)  Now if we
take $K=k-p$ for $f$ of depth~$\le p$, and if $m>p$, then the terms with $s\le p$ vanish because
$(p-m+s)_{m-s}=0$ and those with $s>p$ vanish because $\de^s(f)=0$, so we again find that $(\de-X)^m$ 
annihilates $\Phi_{f,K}(X)$ and therefore that $\de^m$ annihilates $\Phi^{(\d)}_{f,K}(X)$ as desired.

\section{Cocompact and non-cocompact lattices} \label{cocompact}

In this section we discuss a basic dichotomy between the structure of the rings of quasimodular forms 
for non-cocompact and cocompact lattices~$\G$ in~$\SL$.  In the former case, exemplified by the full modular 
group~$\G_1$, there is always a quasimodular but not modular form of weight~2, and then all theorems of the 
previous sections for~$\G_1$ still hold with this form in place of~$E_2$.  In the latter case, exemplified 
by Shimura curves, there is no such quasimodular form and the structure theorems are somewhat different.  
In particular, here there is no analogue of the Serre derivative or the Kaneko-Koike operator, and all 
MLDOs are linear combinations of Rankin-Cohen brackets.

For $\G=\G_1=\MG$ we have already seen that the ring $\qM_*(\G)$ of quasimodular forms is simply $\C[E_2,E_4,E_6]$, 
with the derivations~$\de$ and $D$ given by $12\,\p/\p E_2$ and by Ramanujan's formulas~\eqref{RamDer},
respectively. If $\G$ is a subgroup of $\MG$ of finite index, then the algebra $M_*(\G)$ of modular forms on~$\G$
is in general no longer free, but we still have $\qM_*(\G)=M_*(\G)[E_2]$ and $\de=12\,\p/\p E_2$.  The following 
proposition shows that a similar statement holds for any non-cocompact lattice $\G\subset\SL$.  Recall that for 
any lattice~$\G$, cocompact or not, we have derivations $D$, $W$ and~$\de$ on $\qM_*(\G)$ satisfying~\eqref{commutators}, 
where $W$ is the weight operator and $\,\ker(\de)=M_*(\G)$. In particular, $\de$~maps $\qM_2(\G)$ to~$M_0(\G)=\C$, 
so it must be either 0 or surjective.

\begin{proposition}\label{QMF.Mult} Let $\G\subset\SL$ be an arbitrary lattice.
\newline\noindent  {\rm (a)} If\, $\G$ is cocompact, then $M_2(\G)=\qM_2(\G)$. 
\newline\noindent  {\rm (b)} If\, $\G$ is non-cocompact, then the sequence 
  $0\rightarrow M_2(\G)\rightarrow\qM_2(\Gamma)\xrightarrow{\de}\C\rightarrow0$ is exact.
\newline\noindent  {\rm (c)} In the non-cocompact case, $\qM_*(\G)=M_*(\G)[\E]$ for any $\E\in\qM_2(\G)$ with $\de(\E)\ne0$.
\end{proposition}
\bprf Since the result is certainly known to experts, we only sketch the proof here. 
For~(a), we observe that if $\E$ is a quasimodular form of weight~2 on~$\G$ that is not modular, then $\de\E=C$
for some~$C\ne0$, in which case the same argument as was given for~$\G=\G_1$ in Section~\ref{ERCB} shows
that the ``completion" defined by~$\wh\E(\t)=\E(\t)-\frac C{4\pi y}$ (where again~$y=\Im(\t)$) transforms under~$\G$ 
like a modular form of weight~2.  Then the non-holomorphic 1-form $\omega=\wh\E(\t)\,d\t$ is \hbox{$\G$-invariant} 
and its derivative $d\omega$ is a non-zero multiple of the volume form $y^{-2}d\t\,d\bar\t$ on~$\H$, so the integral 
of~$d\omega$ over~$\H/\G$ is non-zero, which contradicts Stokes's theorem if~$\G$ is cocompact since then~$\H/\G$ 
is closed. For~(b), we observe that if~$\G$ is non-cocompact, then it has at least one cusp, which we can assume 
after conjugating by an element of~$\SL$ to be at~$\infty$. Then the non-holomorphic weight~2 Eisenstein series 
defined by ``Hecke's trick" as $\lim_{\epsilon\to0}\sum(c\t+d)^{-2}|c\t+d|^{-\epsilon}$, where $\sm \cdot \cdot cd$ 
runs over the left cosets of the stabilizer of~$\infty$ in~$\G$, has the form $\E(\t)-\frac C{4\pi y}$ for some 
holomorphic function~$\E$ and constant~$C\ne0$, in which case $\E$~belongs to~$\qM_2(\G)$ but not to~$M_2(\G)$. 
Finally, part~(c) of the proposition follows by an easy induction, since if $F$ is a quasimodular 
form of weight~$k$ and depth~$p>0$, so that the function $G=\de^p(F)$ is a non-zero modular form of weight~$k-2p$,
then the difference between $F$ and some multiple of $G\E^p$ is easily checked to have weight~$k$ and depth~$<p$, 
so that by induction on~$p$ we see that~$F$ is a polynomial in~$\E$ with modular coefficients. 
\eprf

We call a choice of $\E$ in case~(b) a {\it splitting}, since it splits $\qM_2(\G)$ as $M_2(\G)\oplus\C\E$. 
We can, and from now on will, normalize~$\E$ multiplicatively by requiring that~$\de(\E)=1$,
but we then still have the freedom of replacing~$\E$ by $\E+h$ for an arbitrary element $h\in M_2(\G)$.  
For $\G=\MG$ we have $M_2(\G)=\{0\}$, so in that case $\E=\frac1{12}\,E_2$ is unique. In all non-cocompact
cases, if we identify $\qM_*(\G)$ with $M_*(\G)[\E]$, then $\de$ corresponds simply to~$\p/\p\E\,$.

Proposition~\ref{QMF.Mult} gives a complete ``multiplicative" description of the ring of quasimodular forms for 
all non-cocompact groups~$\G$ as polynomials in a single function~$\E\in\qM_2(\G)$ with modular forms as coefficients.  
However, this does not work for cocompact groups, since there is no function~$\E$. 
However, there is also an ``additive" description of~$\qM_*(\G)$ which works for both cocompact and 
non-cocompact groups~$\G$, but which is a little different in the two cases.  

Let~$\G\subset SL_2(\R)$ be an arbitrary lattice and define $\dM$ as the closure of $M_*(\G)$ 
with respect to~$D$, i.e., as the smallest vector space containing $M_*(\G)$ and closed under differentiation.  
The space $\dM$ has the additive structure $\C\oplus\dMp$, where $\dMp$ is the subspace
$$ \dMp\= \C[D]\bigl(M_{>0}(\G)\bigr) \= \C[D]\otimes_{\C}M_{>0}(\G) \= \bigoplus_{n\ge0,\,k>0} D^n(M_k(\G))\,, $$
since $D^n:M_k(\G)\to\dMp$ is injective for all $n\ge0$, $k>0$. Clearly 
\hbox{$M_*(\G)\subset\dM\subseteq\qM_*(\G)$.} The additive description of the space of quasimodular
forms on~$\G$, generalizing equation~\eqref{isom2} for the case of the full modular group, is then as follows.

\begin{proposition}\label{QMF.Add} {\rm (a)} If\, $\G$ is cocompact, then $\qM_*(\G)=\dM$. 
\newline\noindent 
{\rm (b)} If\, $\G$ is non-cocompact, then $\qM_*(\G)=\dM\oplus\C[D]\E$, where $\E$ as in 
Proposition~\rm{\ref{QMF.Mult}\,(b)} is any element in $\wh M_2(\G)\ssm M_2(\G)$. Thus
$$\qM_k(\G)\=\begin{cases} \hphantom{XXXx}  M^D_k(\G) &\quad\text{if $k=0$ or $k$ is odd}, \\
 M^D_k(\G)\,\oplus\,\C\cdot D^n(\E)&\quad \text{if $k=2n+2, \; n\geq0$\,.}\end{cases}$$
\end{proposition}
\begin{proof}   If $F$ is quasimodular of weight~$k>0$ and of depth~$n$, then the
final coefficient $F_n$ in the development~\eqref{qmftransf} is quasimodular of depth~0 and
hence is modular, of weight~$k-2n$.  If $n<k/2$, then $D^n(F_n)$ is quasimodular of the same 
weight and depth as~$F$, so subtracting a multiple of it from~$F$ reduces the depth of~$F$ and
hence proves the result by induction. If~$\G$ is cocompact, then $n$ is always less than~$k/2$,
because $F_n=\de(F_{n-1})/n$ and the map $\de:\qM_2(\G)\to\qM_0(\G)=\C$ is identically~0.
This proves part~(a). If $\G$ is non-cocompact, then $n$ can be equal to~$k/2$, but in that case 
$F_{n-1}$ belongs to $\qM_2(\G)=M_2(\G)\oplus\C\cdot\E$, so we can reduce the depth of~$F$ 
by subtracting from it a multiple of $D^{n-1}(F_{n-1})\in M_k^D(\G)\oplus\C\cdot D^{n-1}(\E)$
and proceed as before. 
\end{proof}

Further properties of the subspace $\dM$ are summarized in the following proposition.
\bp\label{QMF.ideal}
The vector space $\dM$ is an $\lie$-submodule of $\qM_*(\G)$ and is an ideal of the algebra $\qM_*(\G)$.
In particular, $\dM$ is closed under multiplication.  \ep
\bprf It follows easily from the definition of the $\lie$-action on $\qM_*(\G)$ and from Lemma~\ref{GSM.L1}
that $D(D^nF)=D^{n+1}F$, $W(D^nF)=(k+2n)D^nF$ and $\de(D^nF)=n(k+n-1)D^{n-1}F$ for each $F\in\qM_k(\G)$ and $n\geq0$. 
Therefore the subspace $\dM$ is an $\lie$-submodule of $\qM_*(\G)$. This proves the first statement. For the second, 
we define a map $\mu :\qM_*(\G)\rightarrow\C[T]$ by $\mu (F)=0$ if $F\in\qM_k(\G)$ with $k$ odd and 
$\mu (F)=\de^{\,p}(F)\,T^p/p!$ if $k=2p$.  It follows from the Leibniz rule that the map $\mu :\qM_*(\G)\rightarrow\C$
is an algebra homomorphism, so its kernel is an ideal. We claim that this kernel is $\dM$.  The inclusion 
$\dM\subseteq\Ker(\mu )$ is obvious because $M^D_k(\G)=\bigoplus_{0\leq n<k/2}D^nM_{k-2n}(\G)$ for $k>0$ (since 
$D^nM_0(\G)=0$ for $n>0$) and hence $\p^{k/2}$ acts trivially on $M^D_k(\G)$ for $k$ even.  The reverse 
inclusion then follows from the decomposition~$\qM_k=M^D_k\oplus\C\cdot D^{p-1}(\E)$ and 
the easily verified fact that $\mu(D^{p-1}\E)\ne0$. \eprf

The fact that $\dM$ is closed under multiplication means that there must be a formula 
for any product $D^r(f)D^s(g)$ ($r,\,s\ge0$, $f,\,g\in M_*(\G)$) as a linear combination of 
derivatives of modular forms, a simple example being 
  $fg'=\frac\ell{k+\ell}\,(fg)' + \frac1{k+\ell}\,[f,g]_1$ for $f\in M_k$ and $g\in M_\ell$,  
in which the right-hand side contains only modular forms  and their derivatives but no derivatives
of~$\E$. Similarly, in the non-cocompact case, the fact that $\dM$ is an ideal of 
$\qM_*(\G)=\dM\oplus\,\bigoplus_s\C\cdot D^s(\E)$ means that there must also be an expression 
for $D^r(f)D^s(\E)$ as a linear combination of derivatives of modular forms, and the fact that 
$\qM_*(\G)$ is a ring means that there must also be an expression for any product $D^r(\E)D^s(\E)$
as a linear combination of derivatives of both modular forms and~$\E$. These formulas, which are
quite complicated, are given in~\cite{RCA} and will not be repeated here.

Now using this discussion of the structure of quasimodular forms for arbitrary
lattices, we can easily see how all of the theorems proved in this paper have to be modified
when the full modular group~$\G_1$ is replaced by some other lattice~$\G$.  In particular, if $\G$ is
non-cocompact and we have chosen a splitting, then every result discussed or proved so
far remains true {\it mutatis mutandis}.  We state this informally as the following theorem.

\begin{theorem}\label{generalGamma} If $\G\subset\SL$ is a non-cocompact lattice with a given
splitting~$\E$, then all results of Theorems~$\ref{VZtoKK}$--$\ref{MRforQMF}$ remain true with 
$\G_1$ replaced by~$\G$ and $E_2$ by~$12\,\E$.
\end{theorem}
Let us discuss briefly what this statement means in each case.  The first point is that given the 
lattice~$\G$ and the splitting~$\E$ we have a well-defined Serre derivative $\d_k=\d_{k,\E}$ mapping 
$M_k(\G)$ to $M_{k+2}(\G)$ for every~$k\ge0$, defined by the same formula~\eqref{Serrederiv} as before but 
with $E_2(\t)/12$ replaced by~$\E$, or in symbolic notation, by~$\d_\E=D-\E\,W$.  We will usually drop 
the subscript~$\E$ for convenience, but one should not forget that in the general case the new Serre derivative 
is not intrinsic to~$\G$ as it was for the full modular group, but depends on a choice of splitting.  
Replacing $\E$ by $\E^*=\E+h$ with $h\in M_2(\G)$ changes~$\d_\E$ by $h\,W$, i.e., it changes $\d_k$ to 
$\d_k^*(f)=\d_k(f)+khf$.  Similarly, we have a new Kaneko-Koike operator $\K_k^n=\K_{k,\E}^n(f)$ 
from $M_k(\G)$ to $M_{k+2n}(\G)$ given by equation~\eqref{KKdef} with $E_2/12$ replaced by~$\E$,
and also new canonical higher Serre derivatives $\d_k^{[n]}=\d_{k,\E}^{[n]}$ defined by
equation~\eqref{defVZ}, but with $E_4/144$ replaced by the ``curvature" $\Om=\E^2-\E'$, which 
always belongs to~$M_4(\G)$.  Formula~\eqref{rc-bracket-vz} still remains true with these new 
higher Serre derivatives, and shows that the entire structure of~$M_*(\G)$ as a Rankin-Cohen algebra 
is determined by just the multiplication, the new Serre derivative, and the curvature, in accordance 
with the main result of~\cite{Z} (where~$\Om$  was denoted by~$-\Phi$, and a Rankin-Cohen algebra
defined by formulas~\eqref{defVZ} and~\eqref{rc-bracket-vz} from an underlying graded algebra together 
with a derivation~$\d$ of weight~$+2$ and an element~$\Om$ of weight~4 was called a {\it canonical}  
Rankin-Cohen algebra).  With these definitions, the meaning of the generalizations of Theorem~\ref{VZtoKK}
(apart from the specific formulas for the values of $\om_m=\om_{m,\E}\in M_{2m}(\G)$, which will
of course depend on~$\G$ and~$\E$), Theorems~\ref{Struct}, \ref{bcd-Struct}, \ref{atob}, \ref{Monic},
\ref{uniformMLDO}, \ref{VZforQMF} and its corollary, and Propositions~\ref{CK1} and~\ref{prop.ERCB} 
are all clear.  In particular, we have a notion of extended Rankin-Cohen brackets for any non-cocompact
group together with a choice of splitting.  

Finally, we should say a few words about the cocompact case, even though for most of the applications
(in particular, in the theory of VOAs) we only care about the full modular group or its subgroups.
The proofs of Theorems~\ref{a-transf}, \ref{a-h} and~\ref{QMFtoMLDO} 
did not depend in any way on the special properties of~$\G_1$ and were already stated for arbitrary lattices. 
The description of all modular linear differential operators in terms of Rankin-Cohen brackets as given
in Theorem~\ref{Struct} is in principle still valid for any lattice~$\G$, cocompact or not, but with 
the proviso that in the cocompact case all MLDOs have order strictly smaller than half their weight 
and there is no analogue of the Kaneko-Koike operator, so that we only need Rankin-Cohen brackets
with modular forms of strictly positive weight. There is, however, one exception. The 
generalized Serre derivative~$\d_\E(f)=f'-k\E f$ does not involve~$\E$ if~$k=0$, and similarly
$\K_{k,\E}^n(f)$ does not contain~$\E$ if~$k=1-n$, so that the corresponding operators $D$ and $D^n$
are defined even in the cocompact case when no splitting exists.  This corresponds to ``Bol's identity"
$D^n(f|_{1-n}\thin g)=(D^nf)|_{1+n}\thin g$ for any $f\in\Hol(\H)$ and $g\in\SL$, so that $D^n\in\MLDO_{1-n,1-n}(\G)$
for any lattice~$\G$ and any~$n\ge0$.   Apart from this, however,  all MLDOs for non-cocompact
lattices are combinations of Rankin-Cohen brackets with forms of positive weights,  there are no
monic MLDOs with holomorphic quasimodular forms as coefficients, and the dimension of $\MLDO_{k,k+K}(\G)$
is independent of~$k$ and equal to~$\dim\thin\qM_K(\G)\thin$.

\section{Primitive projection and modular linear differential operators}  \label{MLDOandPrimProj}
In Section~\ref{MLDOandQMF} we defined an isomorphism $F\mapsto L_F$ between 
quasimodular forms and MLDOs, and in Section~\ref{ERCB2} the operation of $L_F$ on modular forms
was interpreted as a pairing~\eqref{newpairing} between quasimodular and modular forms.
Here we will define a second isomorphism between quasimodular forms and MLDOs in terms of a certain
projection operator.  This will be done in the opposite order from before, giving first a new pairing 
between quasimodular and modular forms in terms of the action of $\lie$ on the space of quasimodular 
forms and the primitive projection operator for $\lie$-modules, and then generalizing this operator to
non-modular arguments. A reinterpretation of the new isomorphism in terms of almost holomorphic 
modular forms and the holomorphic projection operator will be given in the next section.

\begin{theorem}\label{QMFtoMLDO2} Let $F$ be a quasimodular form of weight~$K$ and depth~$n$ on
an arbitrary lattice~$\G$.  Then for every positive integer $k$ the operator 
\be\label{pairing2}  \L_{F,k} \: 
 \sum_{r=0}^n \frac1{r!}\,\Biggl(\sum_{m=r}^n\frac{(-1)^mD^{m-r}(\de^mF)}{(m-r)!\,(k+K-m-1)_m}\,\Biggr)\,D^r\, \ee
is a modular linear differential operator of order~$n$ and type~$(k,k+K)$, and the map $F\mapsto\L_{F,k}$
gives an isomorphism from $\tM_K^{(\le n)}(\G)$ to $\MLDO^{(\le n)}_{k,k+K}(\G)$.
\end{theorem}
As examples, if $F$ is modular of weight~$K$ (so that $\de^m(F)=0$ for~$m>0$),
then $\L_{F,k}$ is just multiplication by~$F$, while if $F=E_2$ and $k>0$ then $\L_{F,k}$ is
just a multiple of the Serre derivative~$\sd_k\thin$.
Note that the complicated-looking formula~\eqref{pairing2} can be written more simply as
\be\label{pairing3}  \L_{F,k}(f) \= \sum_{m=0}^n\frac{(-1)^m\,D^m(\de^m(F)\thin f)}{m!\,(k+K-m-1)_m}\,, \ee
but in~\eqref{pairing2}  we have expanded by Leibniz's rule to give $\L_{F,k}$ explicitly as a differential operator.

The origin of this new isomorphism is that there is a canonical way to
project any non-negatively graded $\lie$-module (with grading given by~$W$) onto its 
``primitive" part (kernel of~$\de$).  Applied to the ring of quasimodular forms, this 
gives a collection of canonical projection maps~$\pip_k$ from 
$\tM_k(\G)$ to $M_k(\G)$ for all~$k\ge0$ and any lattice~$\G\subset\SL$.
Rather than just stating the formula as a proposition and checking that it works,
we indicate how it can be derived.  We make the Ansatz that $\pip_k$ 
is an element of the universal enveloping algebra of~$\lie$, i.e., that 
it can be represented as a polynomial in the three generators $D,\,W,\,\de$ of~$\lie$.
These elements of course do not commute, but in view of the commutation 
relations~\eqref{commutators} we can write any non-commutative polynomial in 
them as a linear combination of monomials $D^n\de^mW^\ell$ with integers $\ell,\,m,\,n\ge0$.
When we apply them to quasimodular forms of a fixed weight~$k$, the factor $W^\ell$ just 
acts as a known scalar, so we can write the sought-for operator $\pip_k$ as a linear
combination of monomials $D^n\de^m$ with complex coefficients depending only on~$m$, $n$ 
and~$k$. Moreover, since we want $\pip_k$ to preserve the weight, and since $D$ and $\de$ 
increase and decrease the weight by~2, respectively, we must have~$m=n$, so our Ansatz becomes
\be\label{Ansatz}    \pip_k \= \sum_{m\ge0} c_m \,D^m\thin\de^m \ee
with some as yet undetermined coefficients $c_m$ depending on~$k$.  Here the summation can 
be replaced by one over just $0\le m\le k/2$, because $\de^m(f)$ vanishes for $m>k/2$.  Since 
we want $\pip_k$ to be a projection operator onto the subspace $M_k$ of~$\tM_k$, and since 
$\de$ annihilates~$M_k$, we must have $c_0=1$. Furthermore, since we want the image of
$\pip_k$ to be contained in $M_k=\ker(\de)$ we must have $\de\pip_k(f)=0$ for every $f\in\tM_k$.  
Using~\eqref{GSM25} and noting that $\de^m(f)$ has weight~$k-2m$, we calculate 
$$ \de\bigl(\pip_k(f)\bigr) \= \sum_{m\ge0} c_m\bigl(D^m\de\+m(k-m-1)D^{m-1}\bigr)\,\de^m(f)
 \= \sum_{m\ge1} \bigl(c_{m-1}\+m(k-m-1)c_m\bigr)\,D^{m-1}\de^m(f)\,. $$
Equating this to~0 gives the recursion $c_m=-c_{m-1}/m(k-m-1)$ for all~$m\ge1$, which together 
with the initial condition $c_0=1$ determines $c_m$ uniquely as $(-1)^m/m!(k-m-1)_n$.
If $k>2$, then this number is finite for all $0\le m\le k/2$, since $m\le k/2\le k-1$ and
therefore $(k-m-1)_m\ne0$.  Conversely, this calculation shows that the operator~$\pip_k$ defined by
\be\label{pik} 
 \pip_k \: \sum_{0\le m\le k/2} \frac{(-1)^m}{m!\,(k-m-1)_m} \,D^m\de^m 
  \= 1 \m \frac{D\de}{k-2} \+ \frac{D^2\de^2}{2(k-2)(k-3)} \m \cdots  \ee
has the required properties for every~$k>2$, establishing the following result.
\begin{proposition}\label{PrimProj} 
For every lattice~$\G$ and every integer $k>2$ the
operator~\eqref{pik} gives a projection from quasimodular forms of weight~$k$ on~$\G$ 
to the subspace of modular forms of weight~$k$ on~$\G$.
\end{proposition}

We make two remarks.  The first is that Proposition~\ref{PrimProj} 
fails if $k=2$ and we have to define $\pip_2(F)$ instead as $F-\de(F)\thin\E$ with~$\E$ as in Section~\ref{cocompact}
in the non-cocompact case, while in the cocompact case we can simply take~$\pip_2$ to be the identity.  This will not
be important to us since in the application given here the weight will be larger than~2 anyway.  The second is that
the operator~$\pip_k$ defined in~\eqref{pik} satisfies the identity $\pip_k\circ D=0$ as well as $\de\circ\pip_k=0$ by
a calculation exactly similar to the one above. (More specifically, the equation  $\pip_k\circ D=0$ together with the
Ansatz~\eqref{Ansatz} leads to the same recursion $m(k-m-1)c_m=-c_{m-1}$ as before and hence determines~$\pip_k$
uniquely up to a constant.)  This means that we can give an alternative definition of the  map $\pip_k:\qM_k\to M_k$ 
as the projection onto the first factor in the direct sum decomposition $\qM_k=M_k\oplus D(\qM_{k-2})$ for~$k>2$.

Comparing the formulas~\eqref{pik} and~\eqref{pairing3}, we find that the effect of~$L_{F,k}$ on modular forms 
of weight~$k$ is given simply by
\be\label{pairing4} 
  \L_{F,k}(f) \= \pi_{k+K}(fF) \qquad(\thin F\in\qM_K(\G),\;f\in M_k(\G)\thin)\thin,
\ee
because $\de^m(fF)=\de^m(F)f$ for~$f$ modular.
Of course the same formula holds also if $f$ belongs to~$M_k^!(\G)$ (modular forms that are holomorphic in~$\H$
but can have poles at the cusps) or even to~$M_k^{\text{mer}}(\G)$ (meromorphic modular forms of weight~$k$),
and this then proves the modularity of the operator $\L_{F,k}$ and determines it completely, since the map from 
$\MLDO_{k,k+K}(\G)$ to either $\text{Hom}(M_k^!(\G),M_{k+K}^!(\G))\,$ or
 $\text{Hom}(M_k^{\text{mer}}(\G),M_{k+K}^{\text{mer}}(\G))\,$ is injective.
To see that $F\mapsto\L_{F,k}$ is an isomorphism from $\qM_K(\G)$ to~$\MLDO_{k,k+K}(\G)$, at least for~$k>0$
as in the theorem, we calculate the effect on derivatives.  If $F=D^n(g)$ with $g\in M_\ell$ and $K=\ell+2n$,
then equations~\eqref{ERCBexplicit} and~\eqref{compatible} tell us that $L_{F,k}$ is given by
\be\label{pairing4.5}
 L_{D^n(g),k}(f) \= {\binom{n+k-1}n}^{\!\!-1}\thin\bigl[f,\thin g\bigr]_n 
\ee
for any~$f\in M_k$, and a simple calculation that is left to the reader shows that~$\L_{F,k}$ is given by the very similar formula
\be\label{pairing5}
 \L_{D^n(g),k}(f) \= {\binom{n+k+K-2}n}^{\!\!-1}\thin\bigl[f,\thin g\bigr]_n \,.
\ee
for $f\in M_k$ and $g\in M_\ell$ and any~$n>0$.  In view of Theorem~\ref{Struct}, which says that all MLDOs are
given by Rankin-Cohen brackets and the Kaneko-Koike operator, together with the fact that all (holomorphic)
quasimodular forms are linear combinations of derivatives of modular forms or extended derivatives of the 
constant function~1, this completes the proof of Theorem~\ref{QMFtoMLDO2} except in the case when~$\ell=0$ and~$g=1$,
where we have to modify~\eqref{pairing5} to write $L_{D\up n(g),k}(f)$ as a multiple of the extended Rankin-Cohen
bracket~$\langle f,1\rangle_n$ if~$k>0$. If $k$ is 0 or negative, then we have to renormalize~$\L_{F,k}$ in a 
suitable way as already discussed in the case of~$L_{F,k}$ in the remarks following Theorem~\ref{QMFtoMLDO}.
A further remark is that, in virtue of the identity $\pi_k\circ D=0$ noted above, we find that~\eqref{pairing5}
can be generalized by repeated ``integration by parts" to
$$ {\binom{n+k+K-2}n}^{\!\!-1}\,\bigl[f,\thin g\bigr]_n  
  \= (-1)^r \pi_{k+K}\bigl(D^r(f)\thin D^{n-r}(g)\bigr)\qquad\bigl(0\le r\le n\bigr)\,,$$
so that the $n\thin$th Rankin-Cohen brackets can be seen as the result of applying the 
projection operator to any product $D^r(f)D^s(g)$ with~$r+s=n$.

\section{Holomorphic projection and modular linear differential operators}  \label{MLDOandHolProj}

In this final section we recall the bijection between quasimodular forms and almost holomorphic modular forms
and use it to rewrite Theorem~\ref{QMFtoMLDO2} as a statement involving a  holomorphic projection operator 
defined for real-analytic functions in the upper half-plane.  This gives a more conceptual description of the map
from quasimodular forms to MLDOs. The final result, Theorem~\ref{MLDOsViaHolProj}, provides
perhaps the simplest description of MLDOs of all the ones given in this paper.

By definition, an {\it almost holomorphic modular form} of weight~$k$ on a lattice~$\G\subset\SL$ is 
a function $\Phi:\Hh\to\C$ that is a polynomial in~$1/\Im(\t)$ with coefficients that are holomorphic 
functions of moderate growth and that transforms like a modular form of weight~$k$ 
under the operation of~$\G$ on~$\H$. The degree of this polynomial is called
the {\it depth} of~$\Phi$. There is an isomorphism from the space~$\hM_k(\G)$ of all almost holomorphic modular
forms of weight~$k$ on~$\G$ to the space $\qM(\G)$ of quasimodular forms of weight~$k$ on~$\G$ given by 
associating to each function~$\Phi(\t)\in\hM_k(\G)$ its constant term with respect to~$1/\Im(\t)$. (In fact this 
was the original definition of quasimodular forms in~\cite{KZ}, as already mentioned in Section~\ref{MLDOandQMF}.) 
The inverse isomorphism, which is less obvious, maps $F\in\qM_k(\G)$ to its ``completion" 
\be\label{completion}
\wh F(\t) = \sum_{r\ge0}\frac{F_r(\t)}{(2\pi i\thin(\t-\bar\t))^r}\,,
\ee 
where the functions~$F_r$ are defined by~\eqref{qmftransf} or in terms of~$\de$ as~$F_r=\de^rF/r!\thin$. Then the
action of the Lie algebra~$\lie$ on $\qM_*(\G)$ described in Section~\ref{MLDOandQMF} translates into an action on 
$\hM_*(\G)$ by new operators (denoted $\wh D$, $\wh W$ and $\wh\de$ to distinguish them from the corresponding
operators on quasimodular forms) defined by 
\be \label{sl2onAHMF}
   \wh D\thin\Phi(\t)\,=\,\frac1{2\pi i}\,\biggl(\frac{\p\Phi(\t)}{\p\t} 
         \+k\thin\frac{\Phi(\t)}{\t-\bt}\biggr)\,, \;\quad
      \wh W\thin\Phi(\t)\,=\,k\thin\Phi(\t)\,, \;\quad 
     \wh\de\,\Phi(\t)\,=\,2\pi i\thin(\t-\bt)^2\,\frac{\p\Phi(\t)}{\p\bt}
\ee
for $\Phi\in\hM_k(\G)$.  For a more complete discussion of all of this material we refer the reader to Section~5.3 
of~\cite{123}, but with the warning that the notations there are somewhat different (in particular  {$\wh D$, $\wh W$
and~$\wh\de$} are denoted~$\p$, $E$ and~$\delta^*$, respectively) and that there are a few misprints.

We can now use this isomorphism to transfer the projection operator $\pip_k$ as defined by~\eqref{pik} to a 
projection operator~$\pihh_k$ from almost holomorphic modular forms to modular forms of the same weight 
simply by replacing $D$ and~$\de$ by $\wh D$ and~$\wh\de$ in~\eqref{pik}.   Then~\eqref{pairing4} translates 
into the statement that the operator $\L_{F,k}(f)$ defined in Theorem~\ref{QMFtoMLDO2} is equal to
$\pihh_{k+K}(f\wh F)$ if $f$ is a modular form of weight~$k$.  However, this is still not quite what we want because
to obtain $\L_{F,k}$ as an MLDO we need to know its operation on arbitrary differentiable functions
in the upper half-plane, not only on modular forms.  For this we use a {\it holomorphic projection operator} $\pih_k$
that maps real-analytic (or just differentiable)  functions in the upper half-plane to holomorphic functions and that projects
differentiable or real-analytic modular forms of weight~$k\ge2$ on some lattice~$\G\subset\SL$  to holomorphic modular forms
of the same weight and on the same group. There are several different versions of this operator, depending on the space of
functions to which it is applied.  The most special one, and the one that is used most frequently in the theory of modular forms,
is defined directly on the space of differentiable or real-analytic  modular forms~$\Phi$ of weight~$k$ on~$\G$ that are sufficiently
small at infinity and projects them to holomorphic cusp forms of weight~$k$ by the requirement that the Petersson scalar 
product of $\pih_k(\Phi)$ with any holomorphic cusp form $h$ of weight~$k$ on~$\G$ is the same as the Petersson scalar product 
of $\Phi$ with~$h$.  This defines $\pih_k(\Phi)$ uniquely because the Petersson scalar product on the space of cusp forms is 
non-degenerate. At the next level, there is a projection operator defined for arbitrary 1-periodic functions on~$\H$ (i.e., functions on
$\Z\backslash\H$) in terms of their Fourier expansions by the formula 
\begin{align*}
& \pih_k:\;\Phi(x+iy)\,=\,\sum_{n\in\Z}A_n(y)\,e^{2\pi inx} \;\mapsto\; f(z)\,=\,\sum_{n>0} a_n\,e^{2\pi in\t}\,, \\
& \qquad \qquad   a_n \: \frac{(4\pi n)^{k-1}}{(k-2)!}\,\int_0^\infty A_n(y)\,e^{-2\pi ny}\,y^{k-2}\,dy\,,
\end{align*}
which is checked by a straightforward calculation to agree with the previous definition when $\G$ contains the matrix $\sm1101$
and $\Phi$ is modular of weight~$k$ on~$\G$. (Apply the formula $(\pih_k(\Phi),h)=(\Phi,h)$ to the $n$th Poincar\'e series 
$h(\t)=P_{k,n}(\t)=\sum_{\g\in\Z\backslash\G} e^{2\pi in\t}|_k\g$ and use a standard unfolding argument on both sides.)
Finally, both cases can be seen as specializations of a 	yet more general operator that projects {\it arbitrary} smooth
functions in~$\H$ satisfying a suitable growth condition near the boundary to holomorphic functions and that is equivariant
with respect to the action of~$\SL$ and hence sends modular forms to modular forms and 1-periodic functions to
1-periodic functions.  This operator is defined by integrating against the {\it Bergman kernel function} and maps
the function~$\Phi$ to the function $\pih_k(\Phi)$ defined by 
\be\label{Bergman}
\pih_k(\Phi)(\t)\:\frac{k-1}{4\pi}\,\iint_\H \Phi(z)\,\Bigl(\frac{z-\bar z}{\t-\bar z}\Bigr)^k\,dV\qquad(\t\in\H)\,, 
\ee
where $dV$ denotes the invariant volume element~$y^{-2}\thin dx\thin dy$ (with $z=x+iy$ as usual)
and the integral is absolutely convergent by virtue of the growth assumption on~$\Phi$.  The fact that 
the map~$\pih_k$ is $\SL$-equivariant in weight~$k$ follows directly from the transformation property 
$$  \frac{gz\m\overline{gz}}{g\t\m\overline{gz}} \= \frac{c\t+d}{cz+d}\;\frac{z-\bar z}{\t-\bar z}
   \qquad\Bigl(z,\,\t\in\H,\;g=\bm abcd\in\SL\,\Bigr)\,,$$
while the fact that it is the identity on holomorphic functions (and therefore is a projection
operator as claimed, since $\pih_k(\Phi)$ is obviously always holomorphic in~$\t$) is proved by an 
argument involving Cauchy's theorem and integration by parts which we generalize in the following lemma.  

\begin{lemma}\label{almostHolProj}      
If $h(\tau)$ is a holomorphic function in the upper half-plane satisfying suitable growth 
conditions, then for any integer $m$ with $0\le m\le k-2$ we have 
\be\label{AlmostHolProj}
 \pih_k\biggl(\frac{h(\tau)}{(\t-\bar\t)^m}\biggr) \= \frac{(-1)^m}{(k-m-1)_m}\,\frac{d^{\thin m} h(\tau)}{d\t^{\thin m}}\,.
\ee
\end{lemma}
\begin{proof}
We can compute the integral over~$\H$ in~\eqref{Bergman} for~$\Phi(\t)=h(\t)/(\t-\bar\t)^m$ as
 $$ \int_0^\infty\biggl(\int_{\R+iy}\frac{(2i)^{k-m}\,h(z)\,dz}{(\t-z+2iy)^k}\biggr)\,y^{k-m-2}\,dy
   \= 2\pi i\thin\frac{(-1)^k(2i)^{k-m}}{(k-1)!}\,
    \int_0^\infty \frac{d^{k-1}h}{dz^{k-1}}\bigl(\t+2iy\bigr)\,y^{k-m-2}\,dy \,. $$
where we have used the trick that $\bar z=z-2iy$ becomes a {\it holomorphic} 
function of~$z$ when we restrict to the line~$\Im(z)=y$, so that we can apply Cauchy's theorem
to write the inner integral as $2\pi i$ times the residue of its integrand at the unique pole~$z=\tau+2iy$. 
The lemma now follows by $(k-m-2)\thin$-fold integration by parts.
\end{proof}

We can now use Lemma~\ref{almostHolProj} to extend the domain of definition of~$\pih_k$, which was previously 
defined on the space of differentiable functions in the upper half-plane of sufficiently small growth at infinity, to the (not direct!)  sum
of this space with the space  $\AHol(\H)$ of  ``almost holomorphic functions"  in the upper half-plane. Here  ``almost 
holomorphic"  means a polynomial in $1/\Im(\t)$ with holomorphic coefficients, just as it did in the case of  almost holomorphic 
modular forms,  and we no longer need to require any growth condition, because the lemma allows us to define
the weight~$k$ holomorphic projection of any almost holomorphic function  in the upper half-plane
simply by writing it as a finite linear combination of functions $(\t-\bar\t)^{-m}h(\t)$ with $h$ holomorphic
and applying formula~\eqref{AlmostHolProj}, and also tells us that this definition agrees with the one given by the 
Bergman integral whenever the function being projected is small at infinity as well as being almost holomorphic.
This extended definition is applicable in particular to almost holomorphic modular forms, and we have:
\begin{proposition}\label{pisAgree}
The projection maps $\pi_k$ and $\pih_k$ from $\hM_k(\G)$ to $M_k(\G)$ agree.
\end{proposition}
\begin{proof} This follows directly by comparing formulas~\eqref{pik} and~\eqref{AlmostHolProj}, using~\eqref{completion}
and remembering that $D^mh$ is equal to~$(2\pi i)^{-m}\thin d^mh/d\t^m\thin$.
\end{proof}

Proposition~\ref{pisAgree}, combined with equation~\eqref{pairing4} and the bijection $F\mapsto\wh F$ between
quasimodular forms and almost holomorphic modular forms, tells us that $\L_{F,k}(f)\= \pih_{k+K}(f\wh F)$ for
$f\in M_k(\G)$ and $F\in\qM_K(\G)$. (Note that the restriction $m\le k-2$ in Lemma~\ref{almostHolProj} is
not a problem here, since the weight here is $k+K$ and $m$ is bounded by the depth of~$F$ and hence by 
$K/2$, so that the condition is always satisfied when~$k$ is positive, and in fact also for~$k=0$ except
in the case $K=2$ and $\de(F)\ne0$, where $\L_{F,k}$ has to be renormalized by a factor~$k$ and then is
an MLDO even in this case by the remarks following Theorem~\ref{QMFtoMLDO2}.)
But now the discussion above immediately allows us to extend this
statement to arbitrary holomorphic functions $f\thin$: the right-hand side makes sense because~$f\wh F$ is an 
almost holomorphic function, and the maps coincide because both are differentiable operators and agree for 
all modular forms of weight~$k$. Putting everything together, we get the following theorem.

\begin{theorem}\label{MLDOsViaHolProj}  For each positive integer~$k$ there is an isomorphism 
\be\label{isomQMtoMLDO2} 
\tM_*(\G)\; \xrightarrow{\sim}\; \MLDO_{k,k+*}(\G)
\ee
given by associating to $F\in \qM_K(\G)$ the modular linear differential operator~$\L_{F,k}$ defined by
\be\label{QMFandpihol}
\L_{F,k}(f)\= \pih_{k+K}(f\wh F)\;,
\ee
where $\pih_{k+K}$ is the extended holomorphic projection operator defined above.
\end{theorem}

\section*{Acknowledgments}
The first author would like to thank the International Center of Theoretical Physics, Italy, and Max Planck 
Institute for Mathematics, Germany for their support during part of the time that this paper was being written.
The second author was supported in part by JSPS KAKENHI, Grant Numbers JP21K03183, JP18K03215 and JP16H06336.



\end{document}